\documentclass{amsart}
\usepackage{amssymb,amsthm,enumitem,mathdots,mathtools,url,xcolor}
\usepackage{tikz}
\usetikzlibrary{arrows.meta}
\usetikzlibrary{calc}

\theoremstyle{plain}
\newtheorem{theorem}{Theorem}
\newtheorem{corollary}{Corollary}
\newtheorem{lemma}{Lemma}
\newtheorem{proposition}{Proposition}

\theoremstyle{definition}
\newtheorem{definition}{Definition}
\newtheorem{example}{Example}

\theoremstyle{remark}
\newtheorem{remark}{Remark}

\DeclareMathOperator{\Btr}{Btr}
\DeclareMathOperator{\Ptr}{Ptr}
\DeclareMathOperator{\BT}{BT}
\DeclareMathOperator{\PT}{PT}
\DeclareMathOperator{\chr}{char}
\DeclareMathOperator{\tr}{tr}
\DeclareMathOperator{\vc}{vec}

\DeclareMathOperator{\cdiag}{cdiag}
\DeclareMathOperator{\diag}{diag}
\DeclareMathOperator{\rank}{rank}

\allowdisplaybreaks[4]

\begin{document}

\title{Kronecker differences}

\author{Keegan Doig Anderson}
\address[K.D.~Anderson]{
 Department of Mathematics and Applied Mathematics,
 University of Johannesburg,
 Auckland Park, 2006, South Africa
}
\email[K.D.~Anderson]{kdanderson@uj.ac.za}

\author{Yorick Hardy}
\author{Bertin Zinsou}
\address[Y.~Hardy,~B.~Zinsou]{
 School of Mathematics,
 University of the Witwatersrand,
 Johannesburg 2050, South Africa
}
\email[Y.~Hardy]{yorick.hardy@wits.ac.za}
\email[B.~Zinsou]{bertin.zinsou@wits.ac.za}

\address[K.D.~Anderson,Y.~Hardy,~B.~Zinsou]{
 National Institute for Theoretical and Computational Sciences (NITheCS),
 South Africa
}

\begin{abstract}
 Over the real numbers, the Kronecker sum is the unique
 operation on matrices which exponentiates to the Kronecker
 product. Kronecker quotients provide an algebraic view
 of decompositions of matrices in terms of Kronecker products. This article
 explores families of operations, Kronecker differences,
 which are a kind of ``inverse'' for Kronecker sums.
 The correspondence between Kronecker differences and Kronecker quotients is explored.
 Furthermore, we show that a certain class of Kronecker differences may be characterized
 by families of matrices with these families again
 being expressed as Kronecker products. This approach
 provides a different ``nonlinear'' view
 towards tensor decomposition.
\end{abstract}

\maketitle

\section{Introduction}

Let $\mathbb{F}_m$ denote the vector space of $m\times m$ matrices
over the field $\mathbb{F}$ with characteristic $\chr(\mathbb{F})$, where $m\in\mathbb{N}$. We will denote
by $\mathbb{F}_m\otimes\mathbb{F}_n=\mathbb{F}_{mn}$ the tensor
product of $\mathbb{F}_m$ and $\mathbb{F}_n$. By the matrix $E_{ij}$
we mean a matrix in $\mathbb{F}_m$ with a 1 in row $i$ and column $j$ and
0 in every other entry. In general, the order of the matrix $E_{ij}$ will
be clear from the context. The trace $\tr(A)$ of a matrix $A\in\mathbb{F}_m$
is the sum of the main diagonal entries of $A$.

\subsection{Kronecker products and Kronecker quotients}

Let $A\in\mathbb{F}_m$ and $B\in\mathbb{F}_s$. The Kronecker product
\cite{graham81a,horn91,kronecker3}
$A\otimes B\in\mathbb{F}_{ms}$ is the $ms\times ms$ matrix over $F$ with entries (here $:=$ denotes a definition)
\begin{equation}
 \label{eq:kronent}
 (A\otimes B)_{(i-1)s+p,(j-1)s+q} := (A)_{i,j}(B)_{p,q},
\end{equation}
where
 $i,j\in\{1,\ldots,m\}$ and
 $p,q\in\{1,\ldots,s\}$.
In other words, we can write in block matrix form
\begin{equation*}
A\otimes B=\begin{bmatrix}
                    (A)_{1,1}B & (A)_{1,2}B & \ldots & (A)_{1,n} B \\
                    (A)_{2,1}B & (A)_{2,2}B & \ldots & (A)_{2,n} B \\
                    \vdots     & \vdots     & \ddots & \vdots      \\
                    (A)_{m,1}B & (A)_{m,2}B & \ldots & (A)_{m,n} B
           \end{bmatrix}.
\end{equation*}
We note that the Kronecker product may be viewed as a family of operations
on vector spaces of square matrices (indexed by the orders of the matrices involved), 
or as an operation on the graded vector space of all square matrices. In other words,
the Kronecker product is the family of operations
$\otimes:\mathbb{F}_m\times\mathbb{F}_n\to\mathbb{F}_{mn}$ indexed by $m,n\in\mathbb{N}$.

In the same way that a Kronecker product is a family
of operations on various matrix vector spaces (or on the graded vector space of
all square matrices), a Kronecker quotient may be viewed as a family of operations
on various matrix vector spaces (or as a partially defined operation on the graded
vector space of all square matrices). In other words, we consider the family
$\oslash:\mathbb{F}_{mn}\times\mathbb{F}_n\to\mathbb{F}_m$ indexed by $m,n\in\mathbb{N}$
satisfying the following definition.

\begin{definition}[Kronecker quotient \cite{kronquotient,leopardi05a}]
 \label{def:kquotient}%
 A (right) Kronecker quotient is an operation $\oslash$ satisfying
 \begin{equation*}
  (A\otimes B)\oslash B = A
 \end{equation*}
 for all matrices $A\in\mathbb{F}_{m}$ and $B\in\mathbb{F}_{n}$
 ($B$ non-zero) and $m,n\in\mathbb{N}$.
\end{definition}

In general, $(M\oslash A)\otimes A\neq M$ for $mn\times mn$ matrices $M$
and $n\times n$ non-zero matrices $A$. A detailed study of Kronecker quotients
appears in \cite{kronquotient}. Throughout this article, we will consider
only right Kronecker quotients and differences, analogous results hold for
left Kronecker quotients and differences in an obvious way. Uniform Kronecker
quotients are of particular interest, since the family of operations are
related to each other via algebraic equations. More precisely, there is a generating
operation for the entire family in the algebraically obvious way:

\begin{definition}
 \label{def:kquniform}%
 A Kronecker quotient $\oslash$ is a uniform Kronecker quotient if
 \begin{equation*}
  (A\otimes C)\oslash B = A\otimes (C\oslash B)
 \end{equation*}
 for all matrices $A\in\mathbb{F}_{m}$, $B\in\mathbb{F}_n$, $C\in\mathbb{F}_p\otimes\mathbb{F}_n$ and $m,n,p\in\mathbb{N}$,
 and
 \begin{equation*}
  (A+B)\oslash C = (A\oslash C)+(B\oslash C), \qquad
  (kA)\oslash C = k(A\oslash C)
 \end{equation*}
 for all matrices $A$, $B$ and $C$, and scalars $k\in\mathbb{F}$ whenever the operations are defined.
\end{definition}

\begin{remark}
 An equivalent definition for uniform Kronecker quotients
 is as follows, by simply setting $p=1$ in Definition \ref{def:kquniform}:
 A Kronecker quotient $\oslash$ is a uniform Kronecker quotient if
 \begin{math}
  (A\otimes C)\oslash B = A\otimes (C\oslash B)
 \end{math}
 for all matrices $A\in\mathbb{F}_{m}$, $B,C\in\mathbb{F}_n$ and $m,n\in\mathbb{N}$,
 and
 \begin{math}
  (A+B)\oslash C = (A\oslash C)+(B\oslash C),
 \end{math}
 \begin{math}
  (kA)\oslash C = k(A\oslash C)
 \end{math}
 for all matrices $A$, $B$ and $C$, and scalars $k\in\mathbb{F}$ whenever the operations are defined.
\end{remark}

In \cite{kronquotient}, examples of uniform Kronecker quotients over the real
and complex numbers are given. The following lemma establishes the existence
of a uniform Kronecker quotient over an arbitrary field $\mathbb{F}$.

\begin{lemma}\label{lem:kqexist}%
 Let $\mathbb{F}$ be a field.
 Then there exists a uniform Kronecker quotient over $\mathbb{F}$.
\end{lemma}

\begin{proof}
 Let $B,C\in\mathbb{F}_m$ with $C\neq 0$. Then there exists
 an non-zero entry $(C)_{i,j}\neq 0$ in $C$. Now let $f:\mathbb{F}_m\to\mathbb{N}\times\mathbb{N}$
 be defined by $f(C)=(i,j)$ and $f(0_m)=(1,1)$, i.e. $f$ chooses a non-zero entry of $C$, unless
 $C$ is the zero matrix. Define $\text{---}\oslash C$
 as the linear extension of
 \begin{equation*}
  (A\otimes B)\oslash C := A\dfrac{(B)_{f(C)}}{(C)_{f(C)}}
 \end{equation*}
 for all $A\in\mathbb{F}_s$, $s\in\mathbb{N}$. Then
 $\oslash$ is a uniform Kronecker quotient by construction.
\end{proof}

\subsection{Kronecker sums and Kronecker differences}

The Kronecker sum $A\oplus B$ of an $m\times m$ matrix $A$ and an
$n\times n$ matrix $B$ is given by \cite[pp. 268]{horn91}
\begin{equation*}
 A\oplus B = A\otimes I_n + I_m\otimes B.
\end{equation*}
%The non-diagonal entries of $A$ and $B$ are uniquely determined by the Kronecker sum,
%\begin{equation*}
% (A)_{ij} = \tr[(E_{ji}\otimes E_{11})(A\oplus B)],\qquad
% (B)_{ij} = \tr[(E_{11}\otimes E_{ji})(A\oplus B)],
%\end{equation*}
%where $i\neq j$. However, the diagonal entries are not uniquely determined,
%for example $I_{mn}=I_m\oplus 0=0\oplus I_n$.
Over the real and complex numbers, we have
\begin{equation*}
 \exp(A\oplus B) = \exp(A)\otimes\exp(B).
\end{equation*}
The Kronecker sum uniquely satisfies the above equation. The Kronecker sum
also arises in the study of matrix equations of the form
\begin{equation*}
 BX+XA^T = Y
\end{equation*}
where $X$ is the subject of the equation. Taking the vectorization (vec)
on both sides of the equation, and using the fact that $\vc(ABC)=(C^T\otimes A)\vc(B)$,
yields \cite[pp. 268]{horn91}
\begin{equation*}
 (A\oplus B) \vc(X) = \vc(Y).
\end{equation*}

This article defines and studies Kronecker differences and their connection with Kronecker
quotients \cite{kronquotient}.
Similar to the definition of Kronecker quotients, we consider the family $\ominus:\mathbb{F}_{mn}\times\mathbb{F}_n\to\mathbb{F}_m$
indexed by $m,n\in\mathbb{N}$ satisfying the following definition:

\begin{definition}
 \label{def:kqdifference}%
 A Kronecker difference is an operation $\ominus$ satisfying
 \begin{equation*}
  (A\oplus B)\ominus B = A
 \end{equation*}
 for all matrices $A\in\mathbb{F}_{m}$ and $B\in\mathbb{F}_{n}$,
 $m,n\in\mathbb{N}$.
\end{definition}

Every Kronecker quotient induces a Kronecker difference in the following way.

\begin{definition}
 \label{def:kqdefineskd}%
 Let $\oslash$ be a Kronecker quotient. The Kronecker difference $\ominus$ induced by $\oslash$ is the
 Kronecker difference given by
 \begin{equation*}
  M\ominus B := (M-I_m\otimes B)\oslash I_n
 \end{equation*}
 for all $M\in\mathbb{F}_{mn}$, $B\in\mathbb{F}_n$ and $m,n\in\mathbb{N}$.
\end{definition}

\begin{remark}
 \label{rem:kqdefineskd}%
 It is easily verified that this expression satisfies the condition for a
 Kronecker difference, i.e. for every $A\in\mathbb{F}_m$, $B\in\mathbb{F}_n$ and $m,n\in\mathbb{N}$,
 \begin{align*}
  (A\oplus B)\ominus B
   &= (A\otimes I_n + I_m\otimes B)\ominus B \\
   &= (A\otimes I_n + I_m\otimes B - I_m\otimes B)\oslash I_n \\
   &= (A\otimes I_n)\oslash I_n \\
   &= A.
 \end{align*}
\end{remark}

\begin{remark}[Duality 1]
 \label{rem:kddefineskq}%
 Definition \ref{def:kqdefineskd} has a weak dual, namely: Let $\oslash$ be a Kronecker quotient such that
 \begin{equation*}
  M\oslash B = (M(0_m\oplus B)^{-1})\ominus 0_n
 \end{equation*}
 for all $M\in\mathbb{F}_{mn}$, invertible $B\in\mathbb{F}_n$ and $m,n\in\mathbb{N}$,
 where $\ominus$ is a Kronecker difference. Indeed, for invertible $B$,
 $(0_m\oplus B)^{-1} = (I_m\otimes B)^{-1} = I_m\otimes B^{-1}$ and
 \begin{align*}
  (A\otimes B)\oslash B
   &= ((A\otimes I_n)(I_m\otimes B))\oslash B \\
   &= ((A\otimes I_n)(I_m\otimes B)(I_m\otimes B)^{-1})\ominus 0_n \\
   &= (A\oplus 0_n)\ominus 0_n \\
   &= A.
 \end{align*}
\end{remark}

The notion of uniformity in Kronecker quotients has an analogue for Kronecker differences. 
\begin{definition}
 \label{def:kduniform}%
 A Kronecker difference $\ominus$ is a uniform Kronecker difference if
 \begin{equation*}
  (A\oplus C)\ominus B = A\oplus (C\ominus B)
 \end{equation*}
 for all matrices $A\in\mathbb{F}_{m}$, $B\in\mathbb{F}_n$, $C\in\mathbb{F}_p\otimes\mathbb{F}_n$ and $m,n,p\in\mathbb{N}$,
 and
 \begin{equation*}
  (A+B)\ominus (C+D) = (A\ominus C)+(B\ominus D), \qquad
  (kA)\ominus (kC) = k(A\ominus C)
 \end{equation*}
 for all matrices $A$, $B$, $C$ and $D$, and scalars $k\in\mathbb{F}$ whenever the operations are defined.
\end{definition}
Once again, we immediately have a connection between uniform Kronecker quotients and uniform Kronecker differences.% in Proposition \ref{prop:ukqdefinesukd} below.

%Similar to uniform Kronecker quotients, we may take $p=1$ in
%Definition \ref{def:kduniform} by the following lemma.
%
%\begin{lemma}
% A Kronecker difference $\ominus$ is a uniform Kronecker difference if
% \begin{equation*}
%  (A\oplus C)\ominus B = A\oplus (C\ominus B)
% \end{equation*}
% for all matrices $A\in\mathbb{F}_{m}$, $B\in\mathbb{F}_n$, $C\in\mathbb{F}_n$ and $m,n\in\mathbb{N}$,
% and
% \begin{equation*}
%  (A+B)\ominus (C+D) = (A\ominus C)+(B\ominus D), \qquad
%  (kA)\ominus (kC) = k(A\ominus C)
% \end{equation*}
% for all matrices $A$, $B$, $C$ and $D$, and scalars $k\in\mathbb{F}$.
%\end{lemma}
%
%\begin{proof}
% Let $p\in\mathbb{N}$ and $D\in\mathbb{F}_p\otimes\mathbb{F}_n$.
% Then there exists $k\in\mathbb{N}$, matrices $D_{1,1},\ldots, D_{1,k}\in\mathbb{F}_p$ and matrices $D_{2,1},\ldots,D_{2,k}\in\mathbb{F}_n$ such that
% \begin{equation*}
%  D = \sum_{j=1}^k D_{1,j}\otimes D_{2,j}.
% \end{equation*}
% Hence,
% \begin{align*}
%  (A\oplus D)\ominus B
%   &= A\oplus \left(\sum_{j=1}^k D_{1,j}\otimes D_{2,j}\right) \ominus B \\
%   &= \left(A\oplus 0_{pn} + \sum_{j=1}^k 0_m\oplus( D_{1,j}\otimes D_{2,j})\right) \ominus B \\
%   &= (A\oplus 0_{pn})\ominus B + \sum_{j=1}^k (0_m\oplus( D_{1,j}\otimes D_{2,j})) \ominus B \\
%   &= (A\oplus 0_{pn})\ominus B + \sum_{j=1}^k (0_m\oplus( D_{1,j}\otimes D_{2,j})) \ominus B \\
%???
% \end{align*}
%\end{proof}}

\begin{proposition}
 \label{prop:ukqdefinesukd}%
 Let $\oslash$ be a uniform Kronecker quotient.
 Then the induced Kronecker difference $\ominus$ is uniform.
\end{proposition}

\begin{proof}
 Following the definition in Definition \ref{def:kqdefineskd}, %in the following way,
 \begin{equation*}
  M\ominus B := (M-I_m\otimes B)\oslash I_n
 \end{equation*}
 for all $M\in\mathbb{F}_{mn}$, $B\in\mathbb{F}_n$ and $m,n\in\mathbb{N}$,
 where $\oslash$ is a uniform Kronecker quotient.
 In Remark \ref{rem:kqdefineskd} it was shown that $\ominus$ is a Kronecker difference,
 it remains to show that this Kronecker difference is uniform.
 We have for all $A\in\mathbb{F}_{m}$, $B\in\mathbb{F}_n$, and $C\in\mathbb{F}_p\otimes\mathbb{F}_n$, 
 % KDA: I removed the red, because the calculation seems to be correct
 \begin{align*}
  (A\oplus C)\ominus B
   &= (A\otimes I_p\otimes I_n + I_m\otimes C - I_m\otimes I_p\otimes B) \oslash I_n \\
   &= (A\otimes I_p\otimes I_n)\oslash I_n + (I_m\otimes C)\oslash I_n - (I_m\otimes I_p\otimes B) \oslash I_n \\
   &= A\otimes I_p + I_m\otimes(C\oslash I_n) - I_m\otimes I_p\otimes(B\oslash I_n) \\
   &= A\otimes I_p + I_m\otimes[(C\oslash I_n) - I_p\otimes (B\oslash I_n)] \\
   &= A\otimes I_p + I_m\otimes [(C-I_p\otimes B)\oslash I_n] \\
   &= A\otimes I_p + I_m\otimes (C\ominus B) \\
   &= A\oplus (C\ominus B).
 \end{align*}
 Finally, for all $A,B\in\mathbb{F}_{mn}$ and $C,D\in\mathbb{F}_n$,
 \begin{align*}
  (A+B)\ominus (C+D)
   &= (A+B-I_m\otimes C-I_m\otimes D)\oslash I_n \\
   &= (A-I_m\otimes C+ B-I_m\otimes D)\oslash I_n \\
   &= (A-I_m\otimes C)\oslash I_n + (B-I_m\otimes D)\oslash I_n \\
   &= A\ominus C + B\ominus D
 \end{align*}
 and for any scalar $k\in\mathbb{F}$,
 \begin{align*}
  (kA)\ominus (kC)
   &= (kA-I_m\otimes(kC))\oslash I_n \\
   &= [k(A-I_m\otimes C)]\oslash I_n \\
   &= k[(A-I_m\otimes C)\oslash I_n] \\
   &= k(A\ominus C). \qedhere
 \end{align*}
\end{proof}

%Before studying Kronecker differences in more detail, we will need some additional definitions
%and technical results. We present these results in Appendix \ref{sec:orthogonality} and Appendix \ref{sec:lemmata}.
Before studying Kronecker differences in more detail, we will need some additional definitions
and technical results. 
The presentations of these definitions and technicalities are left for Appendix \ref{sec:orthogonality} and Appendix \ref{sec:lemmata}.

\section{Properties of Kronecker sums and Kronecker differences}

Kronecker sums obey the following identities. See for example \cite{schacke04,kronecker3} and their consequences.
\begin{align}
 \tag{S1}\label{eq:S1}
 (A\oplus B)^T &= A^T\oplus B^T \\
 \tag{S2}\label{eq:S2}
 \tr(A\oplus B) &= n\tr A+m\tr B \\
 \tag{S3}\label{eq:S3}
 (kA)\oplus(kB) &= k(A\oplus B) \\
 \tag{S4}\label{eq:S4}
 (A+B)\oplus(C+D) &= (A\oplus C)+(B\oplus D)\\
 \tag{S5}\label{eq:S5}
 (A\oplus B)\oplus C &= A\oplus(B\oplus C)\\
 \tag{S6}\label{eq:S6}
 [A\oplus B,C\oplus D] &= [A,C]\oplus[B,D]\\
 \tag{S7}\label{eq:S7}
 \exp(A\oplus B) &= \exp(A)\otimes\exp(B)
\end{align}

Let $\sigma(A,B):=A\oplus B$. Then \eqref{eq:S3} and \eqref{eq:S4} become
\begin{equation*}
 \sigma(k(A,B))=k\sigma(A,B),\qquad \sigma((A,C)+(B,D))=\sigma(A,C)+\sigma(B,D).
\end{equation*}
In other words, $\sigma$ is linear if and only if \eqref{eq:S3} and \eqref{eq:S4} hold.

We will investigate the corresponding equations for Kronecker differences.
The equations do not hold in general, but each equation holds true if
the left hand argument of $\ominus$ is an appropriate Kronecker difference.
For example, \eqref{eq:D1} below holds when $A=C\oplus B$, i.e.
\begin{equation*}
 ((C\oplus B)\ominus B)^T
  = C^T = (C^T\oplus B^T)\ominus B^T = (C\oplus B)^T\ominus B^T.
\end{equation*}
\begin{align}
 \tag{D1}\label{eq:D1}
 (A\ominus B)^T &= A^T\ominus B^T \phantom{\dfrac1n}\\
 \tag{D2}\label{eq:D2}
 \tr(A\ominus B) &= \frac1n(\tr A-m\tr B) \\
 \tag{D3}\label{eq:D3}
 (kA)\ominus(kB) &= k(A\ominus B)\phantom{\dfrac1n} \\
 \tag{D4}\label{eq:D4}
 (A+B)\ominus(C+D) &= (A\ominus C)+(B\ominus D)\phantom{\dfrac1n}\\
 \tag{D5}\label{eq:D5}
 (A\ominus B)\ominus C &= A\ominus(C\oplus B)\phantom{\dfrac1n}\\
 \tag{D6}\label{eq:D6}
 [A\ominus B,C\ominus D] &= [A,C]\ominus[B,D]\phantom{\dfrac1n}\\
 \tag{D7}\label{eq:D7}
 \exp(A\ominus B) &= \exp(A)\oslash\exp(B) \phantom{\dfrac1n}
\end{align}

Let $\delta(A,B):=A\ominus B$. Thus, $\delta:\mathbb{F}_{mn}\times\mathbb{F}_n\to\mathbb{F}_m$ satisfies
\begin{equation*}
 \delta(A\oplus B,B)=A.
\end{equation*}
Equations \eqref{eq:D3} and \eqref{eq:D4} become
\begin{equation}
 \label{eq:kdlinear}
 \delta(k(A,B))=k\delta(A,B),\qquad \delta((A,C)+(B,D))=\delta(A,C)+\delta(B,D).
\end{equation}
In other words, if $\ominus$ satisfies \eqref{eq:D3} and \eqref{eq:D4} then $\delta$ is linear.
In Section \ref{sec:canon}, we will consider a canonical form for linear $\delta$ and
its consequences for \eqref{eq:D1}, \eqref{eq:D2} and \eqref{eq:D5}.

\begin{proposition}
 \label{prop:induceddiffprop}%
 Let $\oslash$ be a Kronecker quotient and let $\ominus$
 be the induced Kronecker difference
 \begin{equation*}
  X\ominus B:=(X-I_m\otimes B)\oslash I_n
 \end{equation*}
 for all $X\in \mathbb{F}_{mn}$ and $B\in\mathbb{F}_n$.
 \begin{enumerate}[label=(\alph*)]
  \item
   If $(A\oslash B)^T = A^T\oslash B^T$
   %for all $A\in\mathbb{F}_{mn}$ and $B\in\mathbb{F}_n$,\\
   then $(A\ominus B)^T = A^T\ominus B^T$.\phantom{$\dfrac1n$}
   %for all $A\in\mathbb{F}_{mn}$ and $B\in\mathbb{F}_n$.
  \item
   If $\tr(A\oslash B) = \dfrac{\tr(A)}{\tr(B)}$
   %for all $A\in \mathbb{F}_{mn}$ and $B\in\mathbb{F}_n$ with $\tr(B)\neq 0$,\\
   then $\tr(A\ominus B) = \dfrac1n(\tr(A)-m\tr(B))$.
   %for all $A\in \mathbb{F}_{mn}$ and $B\in\mathbb{F}_n$ with $\tr(B)\neq 0$.
  \item
   If $\oslash$ is uniform then $\ominus$ is uniform.\phantom{$\dfrac1n$}
  \item
   If $(kA)\oslash B = k(A\oslash B)$
   %for all $k\in\mathbb{F}$, $A\in\mathbb{F}_{mn}$ and $B\in\mathbb{F}_n$,\\
   then $(kA)\ominus(kB) = k(A\ominus B)$.\phantom{$\dfrac1n$}
   %for all $k\in\mathbb{F}$, $A\in\mathbb{F}_{mn}$ and $B\in\mathbb{F}_n$.
  \item
   If $(A+C)\oslash B = A\oslash B+C\oslash B$
   %for all $A,C\in \mathbb{F}_{mn}$ and $B\in\mathbb{F}_n$,\\
   then $(A+C)\ominus(B+D) = A\ominus B+C\ominus D$.\phantom{$\dfrac1n$}
   %for all $A,C\in \mathbb{F}_{mn}$ and $B,D\in\mathbb{F}_n$.
 \end{enumerate}
\end{proposition}
\begin{proof}
 The proof of (c) appears in Proposition \ref{prop:ukqdefinesukd}, which also proves (d) and (e). The proofs of $(a)$ and $(b)$ each follow
 by straightforward calculation:
 \begin{multline*}
  (A\ominus B)^T = ((A-I_m\otimes B)\oslash I_n)^T = (A-I_m\otimes B)^T\oslash I_n \\
                 = (A^T-I_m\otimes B^T)\oslash I_n = A^T\ominus B^T,
 \end{multline*}
 \begin{multline*}
  \tr(A\ominus B) = \tr((A-I_m\otimes B)\oslash I_n) \\
  = \dfrac{\tr(A-I_m\otimes B)}n
  = \dfrac1n(\tr(A)-m\tr(B)).
  \qedhere
 \end{multline*}
\end{proof}

\begin{remark}[Duality 2]
 To some extent, the duality described in Remark $\ref{rem:kddefineskq}$ is also present here.
 Let $\oslash$ be a Kronecker quotient such that
 \begin{equation*}
  M\oslash B = (M(0_m\oplus B)^{-1})\ominus 0_n
 \end{equation*}
 for all $M\in\mathbb{F}_{mn}$, invertible $B\in\mathbb{F}_n$ and $m,n\in\mathbb{N}$,
 where $\ominus$ is a Kronecker difference.
 If $\operatorname{char}(\mathbb{F})\neq 2$, let $\oslash_s$ be a Kronecker quotient such that
 \begin{equation*}
  M\oslash_s B = \frac12\left[(M(0_m\oplus B)^{-1})\ominus 0_n+((0_m\oplus B)^{-1}M)\ominus 0_n\right]
 \end{equation*}
 for all $M\in\mathbb{F}_{mn}$, invertible $B\in\mathbb{F}_n$ and $m,n\in\mathbb{N}$.
 \begin{enumerate}[label=(\alph*)]
  \item
   If $(kA)\ominus(kB) = k(A\ominus B)$
   %for all $k\in\mathbb{F}$, $A\in\mathbb{F}_{mn}$ and $B\in\mathbb{F}_n$,
   then
     $(kA)\oslash B = k(A\oslash B)$.\phantom{$\dfrac1n$}
     %for all $k\in\mathbb{F}$, $A\in\mathbb{F}_{mn}$ and invertible $B\in\mathbb{F}_n$,
  \item
   If $(A+C)\ominus(B+D) = A\ominus B+C\ominus D$
   %for all $A,C\in \mathbb{F}_{mn}$ and $B,D\in\mathbb{F}_n$,
     then $(A+C)\oslash B = A\oslash B+C\oslash B$.\phantom{$\dfrac1n$}
   %for all $A,C\in \mathbb{F}_{mn}$ and invertible $B\in\mathbb{F}_n$,
  \item
   If %$\operatorname{char}(\mathbb{F})\neq 2$ and 
   $(A\ominus B)^T = A^T\ominus B^T$
   %for all $A\in\mathbb{F}_{mn}$ and $B\in\mathbb{F}_n$, \\
   then $(A\oslash_s B)^T = A^T\oslash_s B^T$.\phantom{$\dfrac1n$}
   %for all $A\in\mathbb{F}_{mn}$ and invertible $B\in\mathbb{F}_n$,
  \item
   If $(kA)\ominus(kB) = k(A\ominus B)$
     %if $\operatorname{char}(\mathbb{F})\neq 2$ then
     then
     $(kA)\oslash_s B = k(A\oslash_s B)$.\phantom{$\dfrac1n$}
     %for all $k\in\mathbb{F}$, $A\in\mathbb{F}_{mn}$ and invertible $B\in\mathbb{F}_n$,
  \item
   If $(A+C)\ominus(B+D) = A\ominus B+C\ominus D$
   %  if $\operatorname{char}(\mathbb{F})\neq 2$ then
     then $(A+C)\oslash_s B = A\oslash_s B+C\oslash_s B$.\phantom{$\dfrac1n$}
   %for all $A,C\in \mathbb{F}_{mn}$ and invertible $B\in\mathbb{F}_n$.
 \end{enumerate}
\end{remark}

\section{A canonical form}
\label{sec:canon}%

Now we turn our attention to Kronecker differences satisfying \eqref{eq:D3} and \eqref{eq:D4}.
It follows that such a Kronecker difference is given, according to \eqref{eq:kdlinear}, by a linear 
map $\delta:(\mathbb{F}_m\otimes\mathbb{F}_n)\times\mathbb{F}_n\to \mathbb{F}_m$. Before the proposition, we remind the reader of the definitions of the block trace in Definition \ref{def:blocktrace} and partial trace in Definition \ref{def:partialtrace} on page%s
\ \pageref{def:blocktrace}. %and \pageref{def:partialtrace} respectively.

\begin{proposition}
 \label{prp:canon}%
 Let $\delta:(\mathbb{F}_m\otimes\mathbb{F}_n)\times\mathbb{F}_n\to \mathbb{F}_m$
 be a linear map which satisfies $\delta(A\oplus B,B)=A$ for all $A\in \mathbb{F}_m$ and $B\in\mathbb{F}_n$.
 Then there exists $\alpha\in \mathbb{F}_m\otimes \mathbb{F}_n\otimes \mathbb{F}_m$
 such that
 \begin{equation*}
  \delta(A,B) = \tr_{12}(\alpha^T(A\otimes I_m-I_m\otimes B\otimes I_m)),\qquad
  \tr_{12}(\alpha^T(X\otimes I_n \otimes I_m)) = X,\vphantom{\sum_{i=1}^m}
 \end{equation*}
 for all $A\in\mathbb{F}_m\otimes\mathbb{F}_n$, $B\in\mathbb{F}_n$ and $X\in \mathbb{F}_m$.
\end{proposition}

\begin{proof}
 A basis for $(\mathbb{F}_m\otimes\mathbb{F}_n)\times\mathbb{F}_n$ is given by
 $(E_{ij}\otimes E_{kl}, 0)$ and $(0,E_{kl})$ for $i,j\in\{1,\ldots,m\}$ and $k,l\in\{1,\ldots,n\}$.
 Then $\delta$ is defined by its action on this basis:
 \begin{equation*}
  \delta(E_{ij}\otimes E_{kl}, 0) := \sum_{r,s=1}^m \alpha_{ij;kl;rs}E_{rs}, \qquad
  \delta(0,E_{kl}) := \sum_{r,s=1}^{m} \beta_{kl;rs}E_{rs}.
 \end{equation*}
 Define
 \begin{equation*}
  \alpha=\displaystyle\sum_{i,j,r,s=1}^m\sum_{k,l=1}^n \alpha_{ij;kl;rs} E_{ij}\otimes E_{kl}\otimes E_{sr}, \qquad
  \beta=\displaystyle\sum_{k,l=1}^n\sum_{r,s=1}^m \beta_{kl;rs} E_{kl}\otimes E_{sr},
 \end{equation*}
 so that
 \begin{equation*}
  \delta(E_{ij}\otimes E_{kl}, 0) = \tr_{12}(\alpha^T(E_{ij}\otimes E_{kl}\otimes I_m)),\qquad
  \delta(0,E_{kl}) = \Btr(\beta^T(E_{kl}\otimes I_m)).\vphantom{\sum_{i=1}^m}
 \end{equation*}
 Since $(A\oplus B,B) = (A\oplus 0,0) + (0\oplus B,B)$,
 the condition
 \begin{equation*}
  \delta(A\oplus B,B) = \delta(A\oplus 0,0) + \delta(0\oplus B,B)
                      = \delta(A\otimes I_n,0) + \delta(I_m\otimes B, B)
                      = A\vphantom{\sum_{i=1}^m}
 \end{equation*}
 reduces by linearity to
 \begin{equation*}
  \delta(E_{ij}\otimes I_n,0)=E_{ij},\qquad \delta(I_m\otimes E_{kl},E_{kl})=0.
 \end{equation*}
 The second equation yields that
 \begin{equation*}
  \delta(I_m\otimes E_{kl},E_{kl})=0
  \quad\iff\quad
  \sum_{i=1}^m \alpha_{ii;kl;rs}+\beta_{kl;rs} = 0.
 \end{equation*}
 It follows that $\beta=-\tr_1(\alpha)$. Writing $A$ and $B$ in the form
 \begin{equation*}
  A = \sum_{i,j=1}^m\sum_{k,l=1}^n a_{ij;kl}E_{ij}\otimes E_{kl},\qquad
  B = \sum_{k,l=1}^n b_{kl}E_{kl}
 \end{equation*}
 yields
 \begin{align*}
  \lefteqn{\delta(A,B)} &\\
              &= \sum_{i,j=1}^m\sum_{k,l=1}^n a_{ij;kl}\delta(E_{ij}\otimes E_{kl},0)
                  + \sum_{k,l=1}^n b_{kl}\delta(0,E_{kl}) \\
              &= \sum_{i,j=1}^m\sum_{k,l=1}^na_{ij;kl}\tr_{12}(\alpha^T(E_{ij}\otimes E_{kl}\otimes I_m))
                  - \sum_{k,l=1}^n b_{kl}\Btr(\tr_1(\alpha)^T(E_{kl}\otimes I_{m})) \\
              &= \sum_{i,j=1}^m\sum_{k,l=1}^na_{ij;kl}\tr_{12}(\alpha^T(E_{ij}\otimes E_{kl}\otimes I_m))
                  - \sum_{k,l=1}^n b_{kl}\tr_{12}(\alpha^T(I_m\otimes E_{kl}\otimes I_{m})) \\
              &= \tr_{12}(\alpha^T (A\otimes I_m - I_m\otimes B\otimes I_m)).\vphantom{\sum_{i=1}^m}
 \end{align*}
 Applying this formula to the condition $\delta(E_{ij}\otimes I_n,0)=E_{ij}$ provides
 \begin{equation*}
  \delta(E_{ij}\otimes I_n,0)=E_{ij}
  \quad\iff\quad
  \tr_{12}(\alpha^T(E_{ij}\otimes I_n \otimes I_m - 0)) = E_{ij}.\vphantom{\sum_{i=1}^m}
 \end{equation*}
 Thus, $\tr_{12}(\alpha^T(X\otimes I_n \otimes I_m)) = X$ for all $A\in \mathbb{F}_m$.
\end{proof}

\begin{remark}
 \label{rem:B=0}%
 When $\delta$ has the form
 \begin{align*}
  \delta(A,B) &= \tr_{12}(\alpha^T(A\otimes I_m-I_m\otimes B\otimes I_m)) \\
              &= \tr_{12}(\alpha^T[(A-I_m\otimes B)\otimes I_m])
               = \delta(A-I_m\otimes B,0)
 \end{align*}
 for all $A\in\mathbb{F}_m\otimes\mathbb{F}_n$ and $B\in\mathbb{F}_n$,
 $\delta$ is completely defined by its action on the subspace
 \begin{equation*}
  \{\, (A,0)\,:\,A\in\mathbb{F}_m\otimes\mathbb{F}_n\,\}.
 \end{equation*}
 Then we may write
 \begin{equation*}
  A\ominus B = (A-I_m\otimes B)\ominus 0.
 \end{equation*}
\end{remark}

% Is it true for linear $\oslash$ that
%  $$A\oslash B = ((I_m\otimes B)^{-1}A)\oslash I_n?$$
%  This needs some sort of homomorphic property?
%  \begin{align*}
%      A\oslash B = [(I_m\otimes B)(I_m\otimes B)^{-1}A]\oslash[BI_n]&=[(I_m\otimes B)\oslash B][((I_m\otimes B)^{-1}A)\oslash I_n] \\
%      &=((I_m\otimes B)^{-1}A)\oslash I_n.
%  \end{align*}

Thus, some of the properties \eqref{eq:D1} -- \eqref{eq:D5}, may be reformulated in terms of the $n\times n$ zero matrix $0_n\in\mathbb{F}_n$,
\begin{gather}
 \tag{D1-0}\label{eq:D1-0}
 (A\ominus 0_n)^T = A^T\ominus 0_n \phantom{\dfrac1n}\\
 \tag{D2-0}\label{eq:D2-0}
 \tr(A\ominus 0_n) = \frac1n(\tr A) \\
% \tag{D3-0}\label{eq:D3-0}
% (kA)\ominus 0_n = k(A\ominus 0_n) \\
% \tag{D4-0}\label{eq:D4-0}
% (A+B)\ominus 0_n = (A\ominus 0_n)+(B\ominus 0_n)\\
 \tag{D5-0}\label{eq:D5-0}
 (A\ominus 0_m)\ominus 0_n = A\ominus 0_{mn}\phantom{\dfrac1n}
% \tag{D6-0}\label{eq:D6-0}
% [A\ominus 0_n,C\ominus 0_n] = [A,C]\ominus 0_n\\
% \tag{D7-0}\label{eq:D7-0}
% \exp(A\ominus 0_n)= \exp(A)\oslash I_n\\
\end{gather}

\begin{lemma}
 \label{lem:zerosimp}%
 If $\ominus:(\mathbb{F}_m\otimes\mathbb{F}_n)\times\mathbb{F}_n\to\mathbb{F}_m$ is such that
 \eqref{eq:D3} and \eqref{eq:D4} hold true. Then,
 \begin{enumerate}[label=(\alph*)]
  \item $(A\ominus B)^T=A^T\ominus B^T$ (for all $A,B$) $\iff$ $(A\ominus 0_n)^T=A^T\ominus 0_n$ \\ (for all $A$),
  \item $\tr(A\ominus B) = \dfrac1n(\tr(A)-m\tr(B))$ (for all $A,B$) $\iff$ $\tr(A\ominus 0_n) = \dfrac1n(\tr A)$       (for all $A$),
  \item \phantom{\rlap{$\dfrac1n$}}$(A\ominus B)\ominus C=A\ominus(C\oplus B)$ (for all $A,B,C$) $\iff$ $(A\ominus 0_m)\ominus 0_n = A\ominus 0_{mn}$ (for all $A$),
 \end{enumerate}
% Then \textup{(D$j$)} holds if and only if
% \textup{(D$j$-0)} holds, for $j\in\{1,2,5\}$.
\end{lemma}

\begin{proof}
 We will show that (D$j$) $\iff$ (D$j$-0).
 Clearly we need only show that (D$j$-0) $\implies$ (D$j$), since (D$j$) $\implies$ (D$j$-0) follows by substituting $B=0$ and $C=0$ in each formula. Each proof follows by Remark \ref{rem:B=0}.\\
 \noindent\eqref{eq:D1-0} $\implies$ \eqref{eq:D1}:
 \begin{align*}
  (A\ominus B)^T
   &= ((A-I_m\otimes B)\ominus 0_m)^T
    = (A-I_m\otimes B)^T\ominus 0_m \\
   &= (A^T-I_m\otimes B^T)\ominus 0_m
    = A^T\ominus B^T.
 \end{align*}

 \noindent\eqref{eq:D2-0} $\implies$ \eqref{eq:D2}:
 \begin{align*}
  \tr(A\ominus B) &= \tr((A-I_m\otimes B)\ominus 0_m) \\
                  &= \dfrac1n\tr(A-I_m\otimes B) = \dfrac1n(\tr A-m\tr B).
 \end{align*}

 \noindent\eqref{eq:D5-0} $\implies$ \eqref{eq:D5}:\\
  For $A\in\mathbb{F}_{pnm}$, $B\in\mathbb{F}_m$ and $C\in\mathbb{F}_n$,
  \begin{align*}
   (A\ominus B)\ominus C
    &= ((A-I_{pn}\otimes B)\ominus 0_m)\ominus C \\
    &= (((A-I_{pn}\otimes B)\ominus 0_m)-I_p\otimes C)\ominus 0_n \\
    &= (((A-I_{pn}\otimes B)\ominus 0_m)-((I_p\otimes C)\oplus 0_m)\ominus 0_m)\ominus 0_n \\
    &= ((A-I_{pn}\otimes B-(I_p\otimes C)\oplus 0_m)\ominus 0_m)\ominus 0_n \\
    &= (A-I_{pn}\otimes B-(I_p\otimes C)\oplus 0_m)\ominus 0_{mn} \\
    &= (A-(I_p\otimes I_n\otimes B+I_p\otimes C\otimes I_m))\ominus 0_{mn} \\
    &= (A-I_p\otimes (C\oplus B))\ominus 0_{mn} \\
    &= A\ominus (C\oplus B). \qedhere
  \end{align*}

% \noindent\eqref{eq:D6-0} $\implies$ \eqref{eq:D6}:\\
% \begin{align*}
%  [A\ominus B,C\ominus D]
%   &= [(A-I_m\otimes B)\ominus 0_m,(C-I_m\otimes D)\ominus 0_m] \\
%   &= [A-I_m\otimes B,C-I_m\otimes D]\ominus 0_m \\
%   &= ([A,C] - I_m\otimes [D,B])\ominus 0_m {\color{red} - [A,I_m\otimes D]\ominus 0_m + [C,I_m\otimes B]\ominus 0_m} \\
%   &= [A,C] \ominus [D,B] {\color{red} - [A,I_m\otimes D]\ominus 0_m + [C,I_m\otimes B]\ominus 0_m.}
% \end{align*}
%
% \noindent\eqref{eq:D7-0} $\implies$ \eqref{eq:D7}:\\
%  \begin{align*}
%   \exp(A\ominus B)
%    &= \exp((A-I_m\otimes B)\ominus 0_m)
%     = \exp(A-I_m\otimes B)\oslash I_n
%     = {\color{red}???}
%  \end{align*} 
\end{proof}

\begin{lemma}
 \label{lem:canon3}%
 Let $\gamma_1\in\mathbb{F}_n$ and $\gamma_0\in\mathbb{F}_m\otimes\mathbb{F}_n\otimes\mathbb{F}_m$
 such that $\tr(\gamma_1)=1$ and $\tr_2(\gamma_0)=0$.
 Then $\delta:(\mathbb{F}_m\otimes\mathbb{F}_n)\times\mathbb{F}_n\to \mathbb{F}_m$ given by
 \begin{equation*}
  \delta(A,B) = \tr_{12}(\alpha^T(A\otimes I_m-I_m\otimes B\otimes I_m)),\qquad
  \alpha:= \sum_{i,j=1}^m E_{ij}\otimes \gamma_1\otimes E_{ji} + \gamma_0
 \end{equation*}
 for all $A\in\mathbb{F}_m\otimes\mathbb{F}_n$ and $B\in\mathbb{F}_n$ satisfies $\delta(A\oplus B,B)=A$.
\end{lemma}

\begin{proof}
 By Proposition \ref{prp:canon}, it suffices to show that $\tr_{12}(\alpha^T(X\otimes I_n \otimes I_m)) = X$
 for all $X\in\mathbb{F}_m$. Let $X=(x_{kl})$ be given by the entries $x_{kl}$ in row $k$ and column $l$
 for $k,l\in\{1,\ldots,m\}$. Using Lemma \ref{lem:tracezero}, we have
 \begin{align*}
  \lefteqn{\tr_{12}(\alpha^T(X\otimes I_n \otimes I_m))}\qquad & \\
   &= \tr_{12}\left(\left(\sum_{i,j=1}^m E_{ji}\otimes \gamma_1^T\otimes E_{ij} + \gamma_0^T\right)
                    \left(\sum_{k,l=1}^m x_{kl}E_{kl}\otimes I_n \otimes I_m\right)\right) \\
   &= \tr_{12}\left(\left(\sum_{k,l=1}^m\sum_{i,j=1}^n x_{kl}\delta_{ik}E_{jl}\otimes \gamma_1^T\otimes E_{ij}\right)\right) \\
   &\qquad + \tr_{12}\left(\gamma_0^T\left(\sum_{k,l=1}^m x_{kl}E_{kl}\otimes I_n \otimes I_m\right)\right) \\
   &= \sum_{k,l=1}^m\sum_{i,j=1}^n \delta_{ik}\delta_{jl}x_{kl}E_{ij}
    = \sum_{k,l=1}^m x_{kl}E_{kl} = X. \qedhere
 \end{align*}
\end{proof}

\begin{theorem}
 \label{thm:canon2}%
 Let $\mathbb{F}$ denote a field
 and $\delta:(\mathbb{F}_m\otimes\mathbb{F}_n)\times\mathbb{F}_n\to \mathbb{F}_m$
 be a linear map satisfying $\delta(A\oplus B,B)=A$ for all $A\in \mathbb{F}_m$ and $B\in\mathbb{F}_n$.
 Then there exists $\gamma\in\mathbb{F}_m\otimes\mathbb{F}_n\otimes\mathbb{F}_m$ such that $\tr_2(\gamma)=0$ and,
 \begin{equation*}
  \delta(A,B) = \tr_{12}(\alpha^T(A\otimes I_m-I_m\otimes B\otimes I_m)),\qquad
  \alpha:= \sum_{i,j=1}^m E_{ij}\otimes E_{11}\otimes E_{ji} + \gamma,
 \end{equation*}
 for all $A\in\mathbb{F}_m\otimes\mathbb{F}_n$ and $B\in\mathbb{F}_n$.
\end{theorem}

\begin{proof}
 By Proposition \ref{prp:canon}, there exists $\alpha\in \mathbb{F}_m\otimes \mathbb{F}_n\otimes \mathbb{F}_m$
 such that
 \begin{equation*}
  \delta(A,B) = \tr_{12}(\alpha^T(A\otimes I_m-I_m\otimes B\otimes I_m)),\qquad
  \tr_{12}(\alpha^T(X\otimes I_n \otimes I_m)) = X,
 \end{equation*}
 for all $A\in\mathbb{F}_m\otimes\mathbb{F}_n$, $B\in\mathbb{F}_n$ and $X\in \mathbb{F}_m$.
 Writing $\alpha$ in the form
 \begin{equation*}
  \alpha = \sum_{i,j=1}^m\sum_{u,v=1}^m E_{ij}\otimes\alpha_{ij;uv}\otimes E_{uv},
  \qquad \alpha_{ij;uv}\in\mathbb{F}_n,
 \end{equation*}
 yields for $X=E_{kl}$,
 \begin{align*}
  \tr_{12}(\alpha^T(E_{kl}\otimes I_n\otimes I_m))
   = \sum_{u,v=1}^m\left(\sum_{i,j=1}^m \delta_{ki}\delta_{lj}\tr(\alpha_{ij;uv})\right)E_{vu}
   = E_{kl}.
 \end{align*}
 Consequently,
 \begin{equation*}
  \tr(\alpha_{kl;uv}) = \delta_{vk}\delta_{ul}.
 \end{equation*}
 It follows that
 \begin{equation*}
  \alpha = \sum_{i,j=1}^m E_{ij}\otimes\alpha_{ij;ji}\otimes E_{ji}
         + \sum_{i,j=1}^m\sum_{u,v=1\atop (i,j)\neq (v,u)}^m E_{ij}\otimes\alpha_{ij;uv}\otimes E_{uv}
 \end{equation*}
 where $\tr(\alpha_{ij;ji})=1$ and $\tr(\alpha_{ij;uv})=0$ for $(i,j)\neq(v,u)$. Now let
 $\gamma_{ij}:=\alpha_{ij;ji}-E_{11}$, so that $\tr(\gamma_{ij})=0$ and noting that
 \begin{equation*}
  \alpha = \sum_{i,j=1}^m E_{ij}\otimes E_{11}\otimes E_{ji} + \gamma
 \end{equation*}
 where
 \begin{equation*}
  \gamma := \sum_{i,j=1}^m E_{ij}\otimes \gamma_{ij}\otimes E_{ji}
     + \sum_{i,j=1}^m\sum_{u,v=1\atop (i,j)\neq (v,u)}^m E_{ij}\otimes\alpha_{ij;uv}\otimes E_{uv}
 \end{equation*}
 completes the proof.
\end{proof}

\begin{remark}
 In the proof above, the choice of the matrix $E_{11}$ was arbitrary, the proof holds
 equally if we substitute any other matrix with unit trace for $E_{11}$.
\end{remark}

\begin{remark}
 \label{rmk:trivial}%
 When $m=1$, Theorem \ref{thm:canon2} reduces to, for all $A,B\in\mathbb{F}_n$,
 \begin{equation*}
  \delta(A,B) = \tr(\alpha^T(A-B))
 \end{equation*}
 where $\tr(\alpha)=1$. When $n=1$, then $\tr_2(\gamma)=\gamma=0$ and $B=b$ is a scalar. Theorem \ref{thm:canon2} reduces to, for all $A\in\mathbb{F}_m$ and $b\in\mathbb{F}$,
 \begin{equation*}
  \delta(A,b) = \tr_{1}(\alpha^T[(A-b I_m)\otimes I_m]),\qquad
  \alpha= \sum_{i,j=1}^m E_{ij}\otimes E_{ji},
 \end{equation*}
 i.e.
 \begin{equation*}
  A\ominus b = \delta(A,b) = \sum_{i,j=1}^m\tr(E_{ij}^T(A-b I_m))E_{ij} = A-bI_m.
 \end{equation*}
\end{remark}

Next we provide a general result similar to Remark \ref{rmk:trivial}, which shows that every Kronecker difference is a difference of (partial) matrix traces, modulo the action of the matrix $\gamma$ in the canonical representation.

\begin{theorem}
 \label{thm:canon}%
 Let $\mathbb{F}$ denote a field with $\chr(\mathbb{F})\nmid n$ and
 let $\delta:(\mathbb{F}_m\otimes\mathbb{F}_n)\times\mathbb{F}_n\to \mathbb{F}_m$
 be a linear map satisfying $\delta(A\oplus B,B)=A$ for all $A\in \mathbb{F}_m$ and $B\in\mathbb{F}_n$.
 Then there exists $\gamma\in\mathbb{F}_m\otimes\mathbb{F}_n\otimes\mathbb{F}_m$ such that $\tr_2(\gamma)=0$ and,
 \begin{equation*}
  \delta(A,B) = \tr_{12}(\alpha^T(A\otimes I_m-I_m\otimes B\otimes I_m)),\qquad
  \alpha:= \frac1n\sum_{i,j=1}^m E_{ij}\otimes I_n\otimes E_{ji} + \gamma,
 \end{equation*}
 for all $A\in\mathbb{F}_m\otimes\mathbb{F}_n$ and $B\in\mathbb{F}_n$. Furthermore,
 \begin{equation*}
  \delta(A,B) = \frac1n(\Ptr(A) - \tr(B)I_m) + \tr_{12}(\gamma^T(A\otimes I_m-I_m\otimes B\otimes I_m)).
 \end{equation*}
\end{theorem}

\begin{proof}
 The proof is identical to the proof for Theorem \ref{thm:canon2}, by setting
 \begin{equation*}
  \gamma_{ij}:=\alpha_{ij;ji}-\dfrac1nI_n.
 \end{equation*}
 Let $A\in\mathbb{F}_m\otimes\mathbb{F}_n$ be written in the form
 \begin{equation*}
  A = \sum_{k,l=1}^m E_{kl}\otimes A_{kl}.
 \end{equation*}
 We have
 \begin{align*}
  \lefteqn{\tr_{12}((\alpha-\gamma)^T(A\otimes I_m-I_m\otimes B\otimes I_m))} \qquad & \\
   &= \frac1n\tr_{12}\left(\sum_{i,j=1}^m ((E_{ij}\otimes I_n)A)\otimes E_{ji} - E_{ij}\otimes B\otimes E_{ji}\right) \\
   &= \frac1n\sum_{i,j,k,l=1}^m \tr((E_{ij}E_{kl})\otimes A_{kl})\otimes E_{ji} - \delta_{ij}\tr(B)E_{ji} \\
   &= \frac1n\sum_{i,j,k,l=1}^m \delta_{jk}\delta_{il}\tr(A_{kl})E_{ji} - \frac1n\sum_{i=1}^m\tr(B)E_{ii} \\
   &= \frac1n\sum_{k,l=1}^m \tr(A_{kl})E_{kl} - \frac1n\tr(B)I_m \\
   &= \frac1n(\Ptr(A) - \tr(B)I_m). \qedhere
 \end{align*}
\end{proof}

\begin{theorem}
 \label{thm:transpose}%
 Let $\alpha\in\mathbb{F}_m\otimes\mathbb{F}_n\times\mathbb{F}_m$ and let
 $\delta:(\mathbb{F}_m\otimes\mathbb{F}_n)\times\mathbb{F}_n\to \mathbb{F}_m$
 be the linear map
 \begin{equation*}
  \delta(A,B) = \tr_{12}(\alpha^T(A\otimes I_m-I_m\otimes B\otimes I_m))
 \end{equation*}
 for all $A\in\mathbb{F}_m\otimes\mathbb{F}_n$ and $B\in\mathbb{F}_n$.
 Then
 \begin{equation*}
  \delta(A^T,B^T)=\delta(A,B)^T\quad\text{for all $A\in\mathbb{F}_m\otimes\mathbb{F}_n$ and $B\in \mathbb{F}_m$}
 \end{equation*}
 if and only if $T_3(\alpha)=T_3(\alpha^T)$.
 Equivalently, if $A\ominus B=\delta(A,B)$,
 \begin{equation*}
  (A\ominus B)^T = A^T\ominus B^T \quad\text{if and only if}\quad T_3(\alpha)=T_3(\alpha^T).
 \end{equation*}
\end{theorem}

\begin{proof}
 By Lemma \ref{lem:zerosimp}, we may assume that $B=0_n$. Now,
 using the identity for the transpose: $A^T=T_3(T_{12}(A))=T_{12}(T_3(A))$, we have
% \begin{align*}
% %\SwapAboveDisplaySkip
%  &\delta(A^T,0_n)=\delta(A,0_n)^T
% \intertext{if and only if}
%  &\tr_{12}(\alpha^T(A^T\otimes I_m))
%   = \tr_{12}(\alpha^T(A\otimes I_m))^T
% \intertext{if and only if}
%  &\tr_{12}(\alpha^T(A^T\otimes I_m))
%   = \tr_{12}(T_3(\alpha^T(A\otimes I_m)))
%   & \text{(by Lemma \ref{lem:parttrans1})} \\
% \intertext{if and only if}
%  &\tr_{12}(\alpha^T(A^T\otimes I_m))
%   = \tr_{12}(T_3(\alpha^T)(A\otimes I_m))
%   & \text{(by Lemma \ref{lem:parttrans3})} \\
% \intertext{if and only if
%  (by the identity $A^T=T_3(T_{12}(A))=T_{12}(T_3(A))$))}
%  &\tr_{12}(T_3(\alpha)(A\otimes I_m))
%   = \tr_{12}(T_3(\alpha^T)(A\otimes I_m))
%  & \text{(by Lemmata \ref{lem:parttrans2}, \ref{lem:parttrans3})}\\
% \intertext{if and only if}
%  &T_3(\alpha) = T_3(\alpha^T).
%   & \text{(by Lemma \ref{lem:parttrequal})}
%   \rlap{\,\,\,\,\quad\qquad\qedhere}
% \end{align*}
  \begin{alignat*}{4}
    \lefteqn{\delta(A^T,0_n)=\delta(A,0_n)^T}\qquad &&& & \\
    &\iff&\quad
    \tr_{12}(\alpha^T(A^T\otimes I_m)) 
    &= \tr_{12}(\alpha^T(A\otimes I_m))^T & \\
    &\iff&\tr_{12}(\alpha^T(A^T\otimes I_m))
    &= \tr_{12}(T_3(\alpha^T(A\otimes I_m))) 
    &&\quad \text{[by Lem. \ref{lem:parttrans1}]} \\
    &\iff&
    \tr_{12}(\alpha^T(A^T\otimes I_m))
    &= \tr_{12}(T_3(\alpha^T)(A\otimes I_m))
    && \quad\text{[by Lem. \ref{lem:parttrans3}]} \\
 %\intertext{if and only if
   %& & & \text{(by the identity $A^T=T_3(T_{12}(A))=T_{12}(T_3(A))$))} \\
    &\iff&
    \tr_{12}(T_3(\alpha)(A\otimes I_m))
    &= \tr_{12}(T_3(\alpha^T)(A\otimes I_m))
    && \quad\text{[by Lems \ref{lem:parttrans2}, \ref{lem:parttrans3}]} \\
    &\iff&
    T_3(\alpha) &= T_3(\alpha^T).
    && \quad\text{[by Lem. \ref{lem:parttrequal}]} &&\quad \qedhere
  \end{alignat*}
\end{proof}

\begin{theorem}
 \label{thm:transpose2}%
 Let $\mathbb{F}$ denote a field and
 let $\gamma\in\mathbb{F}_m\otimes\mathbb{F}_n\otimes\mathbb{F}_m$ be such that $\tr_2(\gamma)=0$.
 Let $\delta:(\mathbb{F}_m\otimes\mathbb{F}_n)\times\mathbb{F}_n\to \mathbb{F}_m$
 be the linear map
 \begin{equation*}
  \delta(A,B) = \tr_{12}(\alpha^T(A\otimes I_m-I_m\otimes B\otimes I_m)), \\\qquad
  \alpha:= \sum_{i,j=1}^m E_{ij}\otimes E_{11}\otimes E_{ji} + \gamma,
 \end{equation*}
 for all $A\in\mathbb{F}_m\otimes\mathbb{F}_n$ and $B\in\mathbb{F}_n$.
 Then
 \begin{equation*}
  \delta(A^T,B^T)=\delta(A,B)^T\quad\text{for all $A\in\mathbb{F}_m\otimes\mathbb{F}_n$ and $B\in \mathbb{F}_m$}
 \end{equation*}
 if and only if $T_3(\gamma)=T_3(\gamma^T)$.
 Equivalently, if $A\ominus B=\delta(A,B)$,
 \begin{equation*}
  (A\ominus B)^T = A^T\ominus B^T \quad\text{if and only if}\quad T_3(\gamma)=T_3(\gamma^T).
 \end{equation*}
\end{theorem}

\begin{proof}
 By Theorem \ref{thm:transpose}, $\delta(A^T,B^T) = \delta(A,B)^T$
 if and only if $T_3(\alpha)=T_3(\alpha^T)$, i.e.
 \begin{equation*}
  \sum_{i,j=1}^m E_{ij}\otimes E_{11}\otimes E_{ij} + T_3(\gamma)
  =
  \sum_{i,j=1}^m E_{ji}\otimes E_{11}\otimes E_{ji} + T_3(\gamma^T),
 \end{equation*}
 so the result follows.
\end{proof}

\begin{theorem}
 Let $\mathbb{F}$ denote a field with $\chr(\mathbb{F})\nmid n$ and
 let $\gamma\in\mathbb{F}_m\otimes\mathbb{F}_n\otimes\mathbb{F}_m$ be such that $\tr_2(\gamma)=0$.
 Let $\delta:(\mathbb{F}_m\otimes\mathbb{F}_n)\times\mathbb{F}_n\to \mathbb{F}_m$
 be the linear map
 \begin{equation*}
  \delta(A,B) = \tr_{12}(\alpha^T(A\otimes I_m-I_m\otimes B\otimes I_m)),\qquad
  \alpha:= \dfrac1n\sum_{i,j=1}^m E_{ij}\otimes I_n\otimes E_{ji} + \gamma,
 \end{equation*}
 for all $A\in\mathbb{F}_m\otimes\mathbb{F}_n$ and $B\in\mathbb{F}_n$.
 Then
 \begin{equation*}
  \delta(A^T,B^T)=\delta(A,B)^T\quad\text{for all $A\in\mathbb{F}_m\otimes\mathbb{F}_n$ and $B\in \mathbb{F}_m$}
 \end{equation*}
 if and only if $T_3(\gamma)=T_3(\gamma^T)$.
 Equivalently, if $A\ominus B=\delta(A,B)$,
 \begin{equation*}
  (A\ominus B)^T = A^T\ominus B^T \quad\text{if and only if}\quad T_3(\gamma)=T_3(\gamma^T).
 \end{equation*}
\end{theorem}

\begin{proof}
 Similar to the proof of Theorem \ref{thm:transpose2}.
% By Theorem \ref{thm:transpose}, $\delta(A^T,B^T) = \delta(A,B)^T$
% if and only if $T_3(\alpha)=T_3(\alpha^T)$, i.e.
% \begin{equation*}
%  \dfrac1n\sum_{i,j=1}^m E_{ij}\otimes I_n\otimes E_{ij} + T_3(\gamma)
%  =
%  \dfrac1n\sum_{i,j=1}^m E_{ji}\otimes I_n\otimes E_{ji} + T_3(\gamma^T),
% \end{equation*}
% so the result follows.
\end{proof}

\begin{theorem}
 Let $\mathbb{F}$ denote a field with $\chr(\mathbb{F})\nmid n$ and
 let $\gamma\in\mathbb{F}_m\otimes\mathbb{F}_n\otimes\mathbb{F}_m$ be such that $\tr_2(\gamma)=0$.
 Let $\delta:(\mathbb{F}_m\otimes\mathbb{F}_n)\times\mathbb{F}_n\to \mathbb{F}_m$
 be the linear map
 \begin{equation*}
  \delta(A,B) = \tr_{12}(\alpha^T(A\otimes I_m-I_m\otimes B\otimes I_m)),\qquad
  \alpha:= \dfrac1n\sum_{i,j=1}^m E_{ij}\otimes I_n\otimes E_{ji} + \gamma,
 \end{equation*}
 for all $A\in\mathbb{F}_m\otimes\mathbb{F}_n$ and $B\in\mathbb{F}_n$.
 Then
 \begin{equation*}
  \tr(\delta(A,B)) = \frac1n(\tr(A) - m\tr(B))\quad\text{for all $A\in\mathbb{F}_m\otimes\mathbb{F}_n$ and $B\in \mathbb{F}_m$}
 \end{equation*}
 if and only if $\tr_3(\gamma)=0$.
 Equivalently, if $A\ominus B=\delta(A,B)$,
 \begin{equation*}
  \tr(A\ominus B) = \dfrac1n(\tr(A) - m\tr(B)) \quad\text{if and only if}\quad \tr_3(\gamma)=0.
 \end{equation*}
\end{theorem}

\begin{proof}
 Direct calculation, by Lemma \ref{lem:zerosimp}, yields
 \begin{align*}
  \tr(\delta(A,0))
   &= \tr(\tr_{12}(\alpha^T(A\otimes I_m))) \\
   &= \tr(\alpha^T(A\otimes I_m)) \vphantom{\smash[b]{\sum_{i,j=1}^m}}\\
   &= \dfrac1n\tr_{12}\left(\tr_3\left(\sum_{i,j=1}^m [(E_{ij}\otimes I_n)A] \otimes E_{ji}\right)\right)
    + \tr(\gamma^T(A\otimes I_m)) \\
   &= \dfrac1n\tr\left(\sum_{i,j=1}^m \delta_{ij}[(E_{ij}\otimes I_n)A]\right)
    + \tr(\gamma^T(A\otimes I_m)) \\
   &= \dfrac1n\tr\left(\sum_{i=1}^m [(E_{ii}\otimes I_n)A]\right)
    + \tr(\gamma^T(A\otimes I_m)) \\
   &= \dfrac1n\tr\left(A\right) + \tr(\gamma^T(A\otimes I_m)). \vphantom{\left(\sum_{i,j=1}^m\right)}
 \end{align*}
 Hence, $\tr(\delta(A,0)) = \frac1n(\tr(A))$ if and only if
 \begin{equation*}
  \tr(\gamma^T(A\otimes I_m)) = 0.
 \end{equation*}
 Since the equation must hold for arbitrary $A$, we must have
 \begin{equation*}
  \tr(\gamma^T(A\otimes I_m)) = 0\quad\text{for all $A\in\mathbb{F}_m\otimes\mathbb{F}_n$}.
 \end{equation*}
 Writing $\gamma$ in the form
 \begin{equation*}
  \gamma = \sum_{i,j=1}^m \gamma_{ij}\otimes E_{ij},
 \end{equation*}
 and observing that
 \begin{equation*}
  \tr_3(\gamma) = \sum_{i,j=1}^m \tr(E_{ij})\gamma_{ij} = \sum_{i=1}^m \gamma_{ii}, \qquad
  \tr(\gamma^T(A\otimes I_m)) = \tr(\tr_3(\gamma^T(A\otimes I_m))),
 \end{equation*}
 shows that
 \begin{equation*}
  \tr(\gamma^T(A\otimes I_m)) = 0\quad\text{for all $A\in\mathbb{F}_m\otimes\mathbb{F}_n$}.
 \end{equation*}
 if and only if
 \begin{equation*}
  \tr\left(\left(\sum_{i=1}^m \gamma_{ii}^T\right)A\right)
   = \tr\left(\tr_3(\gamma)A\right)
   = 0\quad\text{for all $A\in\mathbb{F}_m\otimes\mathbb{F}_n$}.
 \end{equation*}
 Hence, $\tr(\delta(A,B)) = \frac1n(\tr(A) - m\tr(B))$ if and only if $\tr_3(\gamma)=0$.
\end{proof}

\begin{proposition}
 \label{prop:eqvrepr}%
 Let $\mathbb{F}$ denote a field
 and $\delta_1,\delta_2:(\mathbb{F}_m\otimes\mathbb{F}_n)\times\mathbb{F}_n\to \mathbb{F}_m$
 be two linear maps satisfying $\delta_i(A\oplus B,B)=A$ for all $i\in\{1,2\}$, $A\in \mathbb{F}_m$
 and $B\in\mathbb{F}_n$. Suppose that there exists
 $\beta_1,\beta_2\in\mathbb{F}_n$ such that $\tr(\beta_1)=\tr(\beta_2)=1$ and
 $\gamma_1,\gamma_2\in\mathbb{F}_m\otimes\mathbb{F}_n\otimes\mathbb{F}_m$ such that
  $\tr_1(\gamma_1)=\tr_1(\gamma_2)=0$ and
 \begin{align*}
  \delta_1(C,D) &= \tr_{12}(\alpha_1^T(C\otimes I_m-I_m\otimes D\otimes I_m)),\qquad
  \alpha_1:= \sum_{i,j=1}^m E_{ij}\otimes \beta_1\otimes E_{ji} + \gamma_1, \\
  \delta_2(C,D) &= \tr_{12}(\alpha_2^T(C\otimes I_m-I_m\otimes D\otimes I_m)),\qquad
  \alpha_2:= \sum_{i,j=1}^m E_{ij}\otimes \beta_2\otimes E_{ji} + \gamma_2,
 \end{align*}
 for all $C\in\mathbb{F}_m\otimes\mathbb{F}_n$ and $D\in\mathbb{F}_n$.
 Then $\delta_1=\delta_2$ if and only if $(\beta_1,\gamma_1)=(\beta_2,\gamma_2).$
\end{proposition}

\begin{proof}
 We have,
 \begin{equation*}
  \delta_1 = \delta_2
   \,\, \iff \,\,
    \forall C\in\mathbb{F}_m\otimes\mathbb{F}_n,
            D\in\mathbb{F}_m: \,
    \tr_{12}((\alpha_1-\alpha_2)^T(C\otimes I_m - I_m\otimes D\otimes I_m)) = 0.
 \end{equation*}
 We define,
 \begin{align*}
  \Lambda(C,D)
  & \coloneqq
    \tr_{12}((\alpha_1-\alpha_2)^T(C\otimes I_m - I_m\otimes D\otimes I_m)) \\
%  & =
%    \tr_{12}\left(\left(\sum_{i,j=1}^mE_{ij}\otimes (\beta_1-\beta_2)^T\otimes E_{ji}+(\gamma_1-\gamma_2)^T\right)(C\otimes I_m - I_m\otimes D\otimes I_m)\right).
% \end{align*}
% Hence,
% \begin{align*}
%  \Lambda
   & =
    \tr_{12}\left(\left(\sum_{i,j=1}^mE_{ji}\otimes (\beta_1-\beta_2)^T\otimes E_{ij}+(\gamma_1-\gamma_2)^T\right)(C\otimes I_m)\right) \\
   &\qquad
    -\tr_{12}\left(\left(\sum_{i,j=1}^m\left(E_{ji}\otimes
    (\beta_1-\beta_2)^T\otimes E_{ij}\right)(I_m\otimes D\otimes I_m)\right)\right),
  \intertext{
   where we used that $\tr_1(\gamma_1)=\tr_1(\gamma_2)=0$, and furthermore
  }
   \Lambda(C,D)
   & =
    \tr_{12}\left(\sum_{i,j=1}^m\left[\left(E_{ji}\otimes
    (\beta_1-\beta_2)^T\right)C\right]\otimes E_{ij} + (\gamma_1 -
    \gamma_2)^T(C\otimes I_m)\right) \\
   &\qquad
    -\tr_{12}\left(\sum_{i,j=1}^mE_{ji}\otimes\left[(\beta_1 -
    \beta_2)^TD\right]\otimes E_{ij}\right) \\
   & =
    \tr_{12}\left(\sum_{i,j=1}^m\left[\left(E_{ji}\otimes(\beta_1-\beta_2)^T\right)C\right]\otimes E_{ij}
    + (\gamma_1 - \gamma_2)^T(C\otimes I_m)\right) \\
   &\qquad
    - \tr\left((\beta_1 - \beta_2)^TD\right)I_m.
    \phantom{\left(\sum_{i,j=1}^m\right)}
 \end{align*}
 It follows that $\delta_1 = \delta_2$ if and only if
 for all $C\in\mathbb{F}_m\otimes\mathbb{F}_n,
          D\in\mathbb{F}_m$ we have $\Lambda(C,D)=0$.
 Noting that $\tr((\beta_1 - \beta_2)^TD)=0$ for all $D$ if and only if $(\beta_1 - \beta_2)^T=0$, and that $\tr_{12}((\gamma_1 - \gamma_2)^T(C\otimes I_m))=0$ for all $C$ if and only if
 $(\gamma_1 - \gamma_2)^T=0$ (by Lemma \ref{lem:parttrequal}),
 we have that $\delta_1=\delta_2$ if and only if $\beta_1=\beta_2$
 and $\gamma_1=\gamma_2$.
\end{proof}

\section{Uniform Kronecker differences}

In Remark \ref{rmk:trivial}, we noted that if $A$ and $B$ are matrices of the
same order $n\in\mathbb{N}$, then a linear Kronecker difference is given by $A\ominus B=\tr(\upsilon_n^T(A-B))$
for some matrix $\upsilon_n\in \mathbb{F}_n$ with $\tr(\upsilon_n)=1$. These matrices $\upsilon_n$
may be viewed as generators of uniform Kronecker differences.
In Definition \ref{def:kduniform}, it is given that a uniform Kronecker difference
is linear, and hence has a canonical form given in Theorem \ref{thm:canon2}. The
form given in Lemma \ref{lem:canon3} inspires the following lemma.

\begin{lemma}
 \label{lem:unicanon}%
 Let $\upsilon_n\in\mathbb{F}_n$, $n\in\mathbb{N}$ be a sequence of matrices with
 $\tr(\upsilon_n)=1$. Then the Kronecker difference defined below is uniform, where
 \begin{equation*}
  A\ominus B = \tr_{12}(\alpha_{m,n}^T(A\otimes I_m-I_m\otimes B\otimes I_m)),\qquad
  \alpha_{m,n}:= \sum_{i,j=1}^m E_{ij}\otimes\upsilon_n\otimes E_{ji},
 \end{equation*}
 for all $A\in\mathbb{F}_m\otimes\mathbb{F}_n$ and $B\in\mathbb{F}_n$, $m,n\in\mathbb{N}$.
\end{lemma}

\begin{proof}
 By Lemma \ref{lem:canon3}, the operation $\ominus$ is a linear Kronecker difference.
  Let $\sigma_{m,p}$ be the vec permutation matrix \cite{henderson81a} which obeys $\sigma_{m,p}\sigma_{m,p}^T=\sigma_{m,p}^T\sigma_{m,p}=I_{mp}$, and
  \begin{equation*}
   \sigma_{m,p}(A\otimes B)\sigma_{m,p}^T=B\otimes A
  \end{equation*}
  for all $A\in\mathbb{F}_m$ and $B\in\mathbb{F}_p$.
  We have for $A \in \mathbb{F}_m$, $B \in \mathbb{F}_n$, $C \in \mathbb{F}_p \otimes \mathbb{F}_n$, and $m,n \in \mathbb{N}$,
    \begin{align*}
        (A \oplus C) \ominus B &= \delta(A \oplus C, B) \\
        &= \tr_{12}(\alpha_{m,n}^T(A\otimes I_p \otimes I_n \otimes I_{mp} + I_m \otimes C \otimes I_{mp} - I_{mp} \otimes B \otimes I_{mp})) \\
        &= \tr_{12}(\alpha_{m,n}^T(A\otimes I_p \otimes I_n \otimes I_{mp})) \\
        &\qquad + \tr_{12}(\alpha_{m,n}^T(I_m \otimes C \otimes I_{mp} - I_{mp} \otimes B \otimes I_{mp})). 
    \end{align*}
    Then, expanding the two terms on the right hand side yields
    %\begin{align*}
    %    \tr_{12}(\alpha_{m,n}^T(A\otimes I_p \otimes I_n \otimes I_{mp})) 
    %    &= \sum_{i,j=1}^{mp}\tr_{12}((E_{ji}(A\otimes I_p)) \otimes v_n \otimes E_{ij} \\
    %    &= A \otimes I_p
    %\end{align*}
    \begin{equation*}
        \tr_{12}(\alpha_{m,n}^T(A\otimes I_p \otimes I_n \otimes I_{mp})) = \sum_{i,j=1}^{mp}\tr_{12}((E_{ji}(A\otimes I_p)) \otimes v_n \otimes E_{ij} 
        = A \otimes I_p,
    \end{equation*}
    and
    \begin{align*}
      \lefteqn{\tr_{12}(\alpha_{m,n}^T(I_m \otimes C \otimes I_{mp} - I_{mp} \otimes B \otimes I_{mp}))}\\
      &=
        \sum_{i,j=1}^m\sum_{u,v=1}^p
          \tr_{12}\left(
          (E_{ji}\otimes E_{vu}\otimes\upsilon_n^T\otimes E_{ij}\otimes E_{uv})
          (I_m\otimes[C- I_p\otimes B]\otimes I_m\otimes I_p))
          \right) \\
      &=
        \sum_{i,j=1}^m\sum_{u,v=1}^p
          \tr_{12}\left(
          (E_{vu}\otimes\upsilon_n^T\otimes \delta_{ij}E_{ij}\otimes E_{uv})
          ([C- I_p\otimes B]\otimes I_m\otimes I_p))
         \right),
    \end{align*}
    where we used that $\tr(E_{ji})\otimes E_{ij} = \delta_{ij}E_{ij}$. Noting that $\sigma_{m,p}^T\sigma_{m,p}=I_{mp}$,
    \begin{align*}
      \lefteqn{\tr_{12}(\alpha_{m,n}^T(I_m \otimes C \otimes I_{mp} - I_{mp} \otimes B \otimes I_{mp}))}\quad \\
      &=
        \sum_{u,v=1}^p
          \tr_{12}\left(
          (E_{vu}\otimes\upsilon_n^T\otimes I_m\otimes E_{uv})
          ([C- I_p\otimes B]\otimes I_p\otimes I_m))
          \right) \\
      &=
        \sum_{u,v=1}^p
          \tr_{12}\big(
          (I_{p}\otimes I_n\otimes \sigma_{m,p}^T)
          (I_{p}\otimes I_n\otimes \sigma_{m,p})
          (E_{vu}\otimes\upsilon_n^T\otimes I_m\otimes E_{uv}) \\
      & \qquad\times
          ([C- I_p\otimes B]\otimes I_m\otimes I_p))
          (I_{p}\otimes I_n\otimes \sigma_{m,p}^T)
          (I_{p}\otimes I_n\otimes \sigma_{m,p})
          \big) \\
      &=
        \sum_{u,v=1}^p
          \tr_{12}\big(
          (I_{p}\otimes I_n\otimes \sigma_{m,p}^T)
          (E_{vu}\otimes\upsilon_n^T\otimes E_{uv}\otimes I_m) \\
      &\qquad\times
          ([C\otimes I_p- I_p\otimes B\otimes I_p]\otimes I_m))
          (I_{p}\otimes I_n\otimes \sigma_{m,p})
          \big),
    \end{align*}
    since $\sigma_{m,p}(A\otimes B)\sigma_{m,p}^T=B\otimes A$, so that
    \begin{align*}
      \lefteqn{\tr_{12}(\alpha_{m,n}^T(I_m \otimes C \otimes I_{mp} - I_{mp} \otimes B \otimes I_{mp}))}\quad \\
      &=
        \sum_{u,v=1}^p
          \tr_{12}\big(
          (I_{p}\otimes I_n\otimes \sigma_{m,p}^T) \\
      &\qquad\times
          (([E_{vu}\otimes\upsilon_n^T\otimes E_{uv}][C\otimes I_p- I_p\otimes B\otimes I_p])\otimes I_m))
          (I_{p}\otimes I_n\otimes \sigma_{m,p})
          \big)
    \end{align*}
    Hence,
    \begin{align*}
      \lefteqn{\tr_{12}(\alpha_{m,n}^T(I_m \otimes C \otimes I_{mp} - I_{mp} \otimes B \otimes I_{mp}))}\quad \\
      &=
        \sum_{u,v=1}^p
          \sigma_{m,p}^T %\\
      %&\qquad\times
          (\tr_{12}([E_{vu}\otimes\upsilon_n^T\otimes E_{uv}][C\otimes I_p- I_p\otimes B\otimes I_p])\otimes I_m))
          \sigma_{m,p} \\
      &= \sigma_{m,p}^T 
          ((C\ominus B)
          \otimes I_m)
          \sigma_{m,p}.
    \end{align*}
    Using this information, it then follows that
    \begin{align*}
      (A \oplus C) \ominus B &= \delta(A \oplus C, B) \\
      &= \tr_{12}(\alpha_{m,n}^T(A\otimes I_p \otimes I_n \otimes I_{mp}))  \\
      &\qquad + \tr_{12}(\alpha_{m,n}^T(I_m \otimes C \otimes I_{mp} - I_{mp} \otimes B \otimes I_{mp})) \\
      &= A\otimes I_p
       +  \sigma_{m,p}^T 
          ((C\ominus B)
          \otimes I_m)
          \sigma_{m,p} \\
      &= A\otimes I_p+I_m\otimes (C\ominus B) \\
      &= A\oplus (C\ominus B). \qedhere       
    \end{align*}
\end{proof}

\begin{theorem}
 \label{thm:unicanon2}%
 Let $\ominus$ be a uniform Kronecker difference. Then there exists a sequence of matrices
 $\upsilon_n\in\mathbb{F}_n$, $n\in\mathbb{N}$, with $\tr(\upsilon_n)=1$ and matrices
 $\gamma_{m,n}\in\mathbb{F}_{m^2n}$, $m,n\in\mathbb{N}$, with $\tr_{1}(\gamma_{m,n})=0$ and $\tr_{2}(\gamma_{m,n})=0$ and
 \begin{equation*}
  A\ominus B = \tr_{12}(\alpha_{m,n}^T(A\otimes I_m-I_m\otimes B\otimes I_m)),\qquad
  \alpha_{m,n}:= \sum_{i,j=1}^m E_{ij}\otimes\upsilon_n\otimes E_{ji} + \gamma_{m,n},
 \end{equation*}
 for all $A\in\mathbb{F}_m\otimes\mathbb{F}_n$ and $B\in\mathbb{F}_n$, $m,n\in\mathbb{N}$.
\end{theorem}

\begin{proof}
 By Remark \ref{rmk:trivial}, for each $n\in\mathbb{N}$ there exists a matrix $\upsilon_n\in\mathbb{F}_n$
 such that $\tr(\upsilon_n)=1$ and $A\ominus B=\tr(\upsilon_n^T(A-B))$ for all $A,B\in\mathbb{F}_n$.
 Following the proof of Theorem \ref{thm:canon2} (the choice of $E_{11}$ was arbitrary -- we
 required a matrix with unit trace), there exists $\gamma_{m,n}\in\mathbb{F}_{m^2n}$ (dependent on $\upsilon_n$) such that
 $\tr_2(\gamma_{m,n})=0$, and
 \begin{equation*}
  A\ominus B = \tr_{12}(\alpha_{m,n}^T(A\otimes I_m-I_m\otimes B\otimes I_m)),\qquad
  \alpha_{m,n}:= \sum_{i,j=1}^m E_{ij}\otimes\upsilon_n\otimes E_{ji} + \gamma_{m,n}
 \end{equation*}
 for all $A\in\mathbb{F}_m\otimes\mathbb{F}_n$ and $B\in\mathbb{F}_n$, $m,n\in\mathbb{N}$.
 It remains to
 show that $\tr_1(\gamma_{m,n})=0$.
 First let us consider $B,C\in\mathbb{F}_n$.
 Since $\ominus$ is uniform, and following the proof of Lemma \ref{lem:unicanon},
 $(A\oplus C)\ominus B=A\oplus (C\ominus B)$ yields that (recalling that $\tr_2(\gamma_{m,n})=0$ and by use of Lemma \ref{lem:blockpartial})
 \begin{align*}
  \lefteqn{\tr_{12}(\gamma_{m,n}^T(A\otimes I_n\otimes I_m+I_m\otimes C\otimes I_m-I_m\otimes B\otimes I_m))}\qquad & \\
  &= \Btr(\tr_2(\gamma_{m,n})^T(A\otimes I_m))
    + \Btr(\tr_{1}(\gamma_{m,n})^T ((C-B)\otimes I_m)) \\
  &= \Btr(\tr_{1}(\gamma_{m,n})^T ((C-B)\otimes I_m)) \\
  &= 0,
 \end{align*}
 for all $A\in\mathbb{F}_m$, $B,C\in\mathbb{F}_n$ and $m,n\in\mathbb{N}$. Writing $\tr_{1}(\gamma_{m,n})$
 in the form
 \begin{equation*}
  \tr_{1}(\gamma_{m,n}) = \sum_{i,j=1}^n E_{ij}\otimes G_{ij}
 \end{equation*}
 and setting $C-B=E_{kl}$ yields that
 \begin{equation*}
  \sum_{i,j=1}^n \Btr((E_{ji}\otimes G_{ij}^T)(E_{kl}\otimes I_m))
   = \sum_{j=1}^n \Btr(E_{jl}\otimes G_{kj}^T)
   = G_{kl}^T = 0.
 \end{equation*}
 Since $j$ and $k$ were arbitrary, $\tr_{1}(\gamma_{m,n}) = 0$ for all $m,n\in\mathbb{N}.$
 \def\ignore{\color{red}
 We require that
 \begin{equation*}
  (A\oplus C)\ominus B= A\oplus(C\ominus B)=A\oplus 0_p + 0_m\oplus(C\ominus B)   
 \end{equation*}
 for all $m,n,p\in\mathbb{N}$, $A\in\mathbb{F}_m$, $C\in\mathbb{F}_{p}\otimes\mathbb{F}_n$ and $B\in\mathbb{F}_n$.
 By Lemma \ref{lem:unicanon} we need only satisfy
 \begin{gather*}
  \tr_{12}(\gamma_{mp,n}^T(A\otimes I_{pn}\otimes I_{mp}+I_{m}\otimes C\otimes I_{mp}-I_{mp}\otimes B\otimes I_{mp})) \\
  = 0_m\oplus(\tr_{12}(\gamma_{p,n}^T(C\otimes I_p-I_p\otimes B\otimes I_p)))
 \end{gather*}
 for all $A\in\mathbb{F}_m$, $B\in\mathbb{F}_n$ and
 $C\in\mathbb{F}_p\otimes\mathbb{F}_n$. Equivalently, since $\tr_2(\gamma_{mp,n})=0$ (over $\mathbb{F}_{n}$) and $\tr_1(\gamma_{m,n})=0$ (over $\mathbb{F}_m$),
 \begin{equation}
  \label{eq:unigamma}
  \tr_{12}(\gamma_{mp,n}^T(I_{m}\otimes C-I_{mp}\otimes B)\otimes I_{mp})
  = I_m\otimes(\tr_{12}(\gamma_{p,n}^T(C\otimes I_p)))
 \end{equation}
 for all $B\in\mathbb{F}_n$, and $C\in\mathbb{F}_p\otimes\mathbb{F}_n$.
 If $C=0$, we obtain for all $B\in\mathbb{F}_n$
 \begin{equation}
  \label{eq:unigammaB}
  \tr_{12}(\gamma_{mp,n}^T(I_{mp}\otimes B\otimes I_{mp})) = 0_{mp}.
 \end{equation}
 Hence,
 \begin{equation}
  \label{eq:unigammaC}
  \tr_{12}(\gamma_{mp,n}^T(I_{m}\otimes C\otimes I_{mp}))
  = I_m\otimes(\tr_{12}(\gamma_{p,n}^T(C\otimes I_p)))
 \end{equation}
 for all $C\in\mathbb{F}_p\otimes\mathbb{F}_n$.
 First, let us write $\gamma_{mp,n}$ and $\gamma_{p,n}$ in block form, namely
 \begin{equation*}
  \gamma_{p,n} = \sum_{k,l=1}^{pn} E_{kl}\otimes H_{kl},\qquad
  \gamma_{mp,n} = \sum_{i,j=1}^{m}\sum_{k,l=1}^{pn}E_{ij}\otimes E_{kl}\otimes H_{ij;kl}
 \end{equation*}
 where each $H_{kl}\in\mathbb{F}_p$ and each $H_{ij;kl}\in\mathbb{F}_{m}\otimes\mathbb{F}_p$.
 It follows from \eqref{eq:unigammaC} that
 \begin{equation*}
  \sum_{i,j=1}^m\sum_{k,l=1}^{pn}
  \tr(E_{ji})\tr(E_{lk}C)H_{ij;kl}^T
  =
  \sum_{k,l=1}^{pn}\tr(E_{lk}C)I_m\otimes H_{kl}^T
 \end{equation*}
 for all $C\in\mathbb{F}_p\otimes\mathbb{F}_n$.
 If $C=E_{kl}\in\mathbb{F}_{pn}$, then
 \begin{equation*}
  \sum_{i=1}^m H_{ii;kl}^T= I_m\otimes H_{kl}^T.
 \end{equation*}
 In particular, if $p=1$ we obtain
 \begin{equation*}
  \tr_1(\gamma_{m,n}) = \gamma_{1,n}\otimes I_m.
 \end{equation*}
 Now \eqref{eq:unigammaB} becomes
 \begin{equation*}
  \tr(\gamma_{1,n}^TB)\otimes I_{mp} = 0_{mp}
 \end{equation*}
 for all $B\in\mathbb{F}_n$. Hence, $\gamma_{1,n}=0$ and $\tr_1(\gamma_{m,n})=0$ for all $m,n\in\mathbb{N}$.
 }
\end{proof}

Now, by Proposition \ref{prop:eqvrepr}, it follows that for a given sequence of $n\times n$ matrices $(\upsilon_n)$, the matrices $\gamma_{m,n}$ uniquely define a uniform Kronecker difference.

\begin{corollary}
 \label{cor:uniqueuniform}%
 Let $\ominus$ be a uniform Kronecker difference with a given
 sequence $(\upsilon_n)$ of $n\times n$ matrices with $\tr(\upsilon_n)=1$ as in Theorem \ref{thm:unicanon2}. Then there exist unique matrices
 $\gamma_{m,n}\in\mathbb{F}_{m^2n}$, $m,n\in\mathbb{N}$, with $\tr_{1}(\gamma_{m,n})=0$ and $\tr_{2}(\gamma_{m,n})=0$ and
  \begin{equation*}
  A\ominus B = \tr_{12}(\alpha_{m,n}^T(A\otimes I_m-I_m\otimes B\otimes I_m)),\qquad
  \alpha_{m,n}:= \sum_{i,j=1}^m E_{ij}\otimes\upsilon_n\otimes E_{ji} + \gamma_{m,n},
 \end{equation*}
 for all $A\in\mathbb{F}_m\otimes\mathbb{F}_n$ and $B\in\mathbb{F}_n$, $m,n\in\mathbb{N}$.
\end{corollary}

We note that by Lemma \ref{lem:unicanon}, we obtain the following corollary.
\begin{corollary}
 Let $\ominus$ be a uniform Kronecker difference given by the pairs $(\upsilon_n,\gamma_{m,n})$ of matrices,
 $m,n\in\mathbb{N}$, as given in Corollary \ref{cor:uniqueuniform}.
 Then the pairs of matrices $(\upsilon_n,0)$ yield
 a uniform Kronecker difference $\ominus_0$.
\end{corollary}

Although the matrices $\gamma_{m,n}$ in Corollary \ref{cor:uniqueuniform} uniquely define uniform
Kronecker differences,
associativity of the Kronecker sum \eqref{eq:S5} reflects
in the sequence of matrices $(\upsilon_n)$.

\begin{theorem}
 \label{thm:assoc}%
 Let $\ominus$ be a uniform Kronecker difference given by the pairs $(\upsilon_n,\gamma_{m,n})$ of matrices,
 $m,n\in\mathbb{N}$, as given in Corollary \ref{cor:uniqueuniform}.
 If the following condition holds for all $X\in\mathbb{F}_m\otimes\mathbb{F}_p\otimes\mathbb{F}_q$,
 $Y\in\mathbb{F}_q$, $Z\in\mathbb{F}_p$ and $m,p,q\in\mathbb{N}$
 \begin{equation*}
  (X\ominus Y)\ominus Z = X\ominus (Z\oplus Y),
 \end{equation*}
 then for all $p,q\in\mathbb{N}$:
 \begin{equation*}
  \Btr(\upsilon_{pq})=\upsilon_q, \qquad \Ptr(\upsilon_{pq})=\upsilon_p.
 \end{equation*}
\end{theorem}

\begin{proof}
 Let $p,q\in\mathbb{N}$.
 By Lemma \ref{lem:zerosimp}(c), it suffices to show that
 the following condition holds: if for all $m\in\mathbb{N}$ and $X\in\mathbb{F}_{m}\otimes\mathbb{F}_p\otimes\mathbb{F}_q$,
 \begin{equation*}
  (X\ominus 0_q)\ominus 0_p = X\ominus (0_p\oplus 0_q) = X\ominus 0_{pq},
 \end{equation*}
 then:
 \begin{equation*}
  \Btr(\upsilon_{pq})=\upsilon_q, \qquad \Ptr(\upsilon_{pq})=\upsilon_p.
 \end{equation*}
 
 To obtain the result $\Btr(\upsilon_{pq})=\upsilon_q$, it suffices to consider $X=I_m\otimes I_p\otimes Y$ where $Y\in\mathbb{F}_q$.
 We have
 \begin{align*}
  X\ominus 0_{pq}
   &= \tr_{12}(\alpha_{m,pq}^T(X\otimes I_m)) \\
   &= \sum_{i,j=1}^m \tr((E_{ji}\otimes\upsilon_{pq}^T)(I_m\otimes I_p\otimes Y))E_{ij}
    + \tr_{12}(\gamma_{m,pq}^T(I_m\otimes I_p\otimes Y\otimes I_m)) \\
   &= \sum_{i,j=1}^m \tr((E_{ji}\otimes\upsilon_{pq}^T)(I_m\otimes I_p\otimes Y))E_{ij} \\[5pt]
   &= \tr(\upsilon_{pq}^T(I_p\otimes Y))I_m,
 \end{align*}
 using Lemma \ref{lem:trzidz} and the fact that $\tr_1(\gamma_{m,pq}^T)=0$. On the other hand,
 \begin{align*}
  (X\ominus 0_q)\ominus 0_p
    &= \left[\tr_{12}(\alpha_{mp,q}^T(I_m\otimes I_p\otimes Y\otimes I_{mp}))\right]\ominus 0_p \phantom{\left[\sum_{i,j=1}^{mnp}\right]}\\
%    &= \left(\sum_{i,j=1}^{mp}\tr((E_{ji}\otimes\upsilon_{q}^T)(I_m\otimes I_p\otimes Y))E_{ij}\right)\ominus 0_p\\
%    &\qquad +\left(\tr_{12}(\gamma_{mp,q}^T(I_m\otimes I_p\otimes Y\otimes I_{mp})\right)\ominus 0_p \\
    &= \left[\sum_{i,j=1}^{mp}\tr((E_{ji}\otimes\upsilon_{q}^T)(I_{mp}\otimes Y))E_{ij}\right]\ominus 0_p\\
    &= \tr(\upsilon_q^TY)(I_{mp}\ominus 0_p),\phantom{\left[\sum_{i,j=1}^{mnp}\right]}\\
 \intertext{where we used that $\displaystyle\sum_{i,j=1}^{mp}\tr(E_{ji})=1,$ so that}
  (X\ominus 0_q)\ominus 0_p
    &= \tr(\upsilon_q^TY)\tr_{12}(\alpha_{m,p}^T (I_{mp}\otimes I_m))\phantom{\sum_{i,j=1}^{mnp}}\\
    &= \tr(\upsilon_q^TY)\left[\sum_{i,j=1}^{m}\tr(E_{ji})\tr(\upsilon_{p}^T)E_{ij}\right]\\
    &= \tr(\upsilon_{q}^TY)\tr(\upsilon_p^T)I_m \phantom{\left[\sum_{i,j=1}^{mnp}\right]}\\
    &= \tr((\upsilon_p\otimes\upsilon_{q})^T(I_p\otimes Y))I_m,
    \phantom{\left[\sum_{i,j=1}^{mnp}\right]}
 \end{align*}
 where we again used Lemma \ref{lem:trzidz}, $\tr_1(\gamma_{mp,q}^T)=0$ and $\tr_1(\gamma_{m,p}^T)=0$.
 Hence, by Lemma \ref{lem:Btrequiv}, $\Btr(\upsilon_{pq})=\Btr(\upsilon_p\otimes\upsilon_q)=\tr(\upsilon_p)\upsilon_q=\upsilon_q.$

 To obtain the result $\Ptr(\upsilon_{pq})=\upsilon_p$, it suffices to consider $X=I_m\otimes Z\otimes I_q$ where $Z\in\mathbb{F}_p$.
 We have
 \begin{align*}
  X\ominus 0_{pq}
   &= \tr_{12}(\alpha_{m,pq}^T(X\otimes I_m))\phantom{\sum_{i,j=1}^m}\\
   &= \sum_{i,j=1}^m \tr((E_{ji}\otimes\upsilon_{pq}^T)(I_m\otimes Z\otimes I_q))E_{ij} \\
    & \qquad + \sum_{i,j=1}^m \tr_{12}(\gamma_{m,pq}^T(I_m\otimes Z\otimes I_q\otimes I_m)) \\
   &= \sum_{i,j=1}^m \tr((E_{ji}\otimes\upsilon_{pq}^T)(I_m\otimes Z\otimes I_q))E_{ij} \\
   &= \tr(\upsilon_{pq}^T(Z\otimes I_q))E_{ij},\phantom{\sum_{i,j=1}^m}
 \end{align*}
 using Lemma \ref{lem:trzidz} and the fact that $\tr_1(\gamma_{m,pq})=0$. Now,
 \begin{align*}
  (X\ominus 0_q)\ominus 0_p
    &= \left[\tr_{12}(\alpha_{mp,q}^T(I_m\otimes Z\otimes I_q\otimes I_{mp}))\right]\ominus 0_p
    \phantom{\sum_{i,j=1}^m} \\
    &= \left[\sum_{i,j=1}^{mp}\tr((E_{ji}\otimes\upsilon_{q}^T)(I_m\otimes Z\otimes I_q))E_{ij}\right]\ominus 0_p\\
    &= \tr(\upsilon_q^T)\left[\sum_{i,j=1}^{mp}\tr(E_{ji}(I_m\otimes Z))E_{ij}\right]\ominus 0_p
    \phantom{\sum_{i,j=1}^m} \\
    &= \tr(\upsilon_q^T)(I_m\otimes Z)\ominus 0_p
    \phantom{\sum_{i,j=1}^m} \\
    &= \tr(\upsilon_q^T)\tr_{12}(\alpha_{m,p}^T (I_m\otimes Z\otimes I_m))
    \phantom{\sum_{i,j=1}^m} \\
    &= \tr(\upsilon_q^T)\left[\sum_{i,j=1}^{m}\tr(E_{ji})\tr(\upsilon_{p}^TZ)E_{ij}\right]\\
    &= \tr(\upsilon_{q}^T)\tr(\upsilon_p^TZ)I_m 
    \phantom{\sum_{i,j=1}^m} \\
    &= \tr((\upsilon_p\otimes\upsilon_q)^T(Z\otimes I_q))I_m.
    \phantom{\sum_{i,j=1}^m}
 \end{align*}
 Hence, by Lemma \ref{lem:Btrequiv}, $\Ptr(\upsilon_{pq})=\Ptr(\upsilon_p\otimes\upsilon_q)=\tr(\upsilon_q)\upsilon_p=\upsilon_p.$
\end{proof}

Let us consider a special case for Theorem \ref{thm:assoc}, namely the case where each $\gamma_{m,n}=0$.
\begin{proposition}
 \label{prop:assoc}%
 Let $\ominus$ be a uniform Kronecker difference given by the pairs $(\upsilon_n,0)$ of matrices,
 $m,n\in\mathbb{N}$, as given in Corollary \ref{cor:uniqueuniform}.
 If the following condition holds for all $X\in\mathbb{F}_m\otimes\mathbb{F}_p\otimes\mathbb{F}_q$,
 $Y\in\mathbb{F}_q$, $Z\in\mathbb{F}_p$ and $m,p,q\in\mathbb{N}$
 \begin{equation*}
  (X\ominus Y)\ominus Z = X\ominus (Z\oplus Y),
 \end{equation*}
 then for all $p,q\in\mathbb{N}$:
 \begin{equation*}
  \upsilon_{pq} = \upsilon_p\otimes\upsilon_q = \upsilon_q\otimes\upsilon_p.
 \end{equation*}
\end{proposition}

\begin{proof}
 Let $p,q\in\mathbb{N}$.
 By Lemma \ref{lem:zerosimp}(c), it suffices to show that
 the following condition holds: if for all $m\in\mathbb{N}$ and $X\in\mathbb{F}_{m}\otimes\mathbb{F}_p\otimes\mathbb{F}_q$,
 \begin{equation*}
  (X\ominus 0_q)\ominus 0_p = X\ominus (0_p\oplus 0_q) = X\ominus 0_{pq},
 \end{equation*}
 then:
 \begin{equation*}
  \upsilon_{pq}=\upsilon_p\otimes\upsilon_q.
 \end{equation*}
 We have
 \begin{align*}
  X\ominus 0_{pq}
   &= \tr_{12}(\alpha_{m,pq}^T(X\otimes I_m))
   \phantom{\sum_{i,j=1}^m}\\
   &= \sum_{i,j=1}^m \tr_{12}((E_{ji}\otimes \upsilon_{pq}^T\otimes E_{ij})(X\otimes I_m)) \\
   &= \sum_{i,j=1}^m \tr((E_{ji}\otimes \upsilon_{pq}^T)X)E_{ij}.
 \end{align*}
 On the other hand, noting that (from the context below) $E_{ij}$ will denote an $m\times m$ matrix and $E_{kl}$ will
 denote an $(mp)\times(mp)$ matrix,
 \begin{align*}
  \lefteqn{(X\ominus 0_q)\ominus 0_p}\quad&\\
    &= \left[\tr_{12}(\alpha_{mp,q}^T(X\otimes I_{mp}))\right]\ominus 0_p
    \phantom{\sum_{i,j=1}^m} \\
    &= \left[\sum_{k,l=1}^{mp}\tr_{12}((E_{lk}\otimes\upsilon_{q}^T\otimes E_{kl})(X\otimes I_{mp}))\right]\ominus 0_p \\
    &= \sum_{k,l=1}^{mp}\tr((E_{lk}\otimes\upsilon_{q}^T)X)\tr_{12}(\alpha_{m,p}^T (E_{kl}\otimes I_m))
\intertext{since $\tr_{12}((E_{lk}\otimes\upsilon_{q}^T)X\otimes E_{kl})=\tr((E_{lk}\otimes\upsilon_{q}^T)X)E_{kl}$, whence}
  \lefteqn{(X\ominus 0_q)\ominus 0_p}\quad&\\
    &= \sum_{i,j=1}^m\sum_{k,l=1}^{mp}\tr((E_{lk}\otimes\upsilon_{q}^T)X)\tr_{12}((E_{ji}\otimes \upsilon_p^T\otimes E_{ij})(E_{kl}\otimes I_m)) \\
    &= \sum_{i,j=1}^m\sum_{k,l=1}^{mp}\tr((E_{lk}\otimes\upsilon_{q}^T)X)\tr((E_{ji}\otimes \upsilon_p^T)E_{kl})E_{ij} \\
    &= \sum_{i,j=1}^m\sum_{s,t=1}^{m}\sum_{u,v=1}^p\tr((E_{ts}\otimes E_{vu}\otimes\upsilon_{q}^T)X)\tr((E_{ji}\otimes \upsilon_p^T)(E_{st}\otimes E_{uv}))E_{ij},\\
 \intertext{where for every $k,l\in\{1,\ldots,mp\}$ there exists unique $s,t\in\{1,\ldots,m\}$ and $u,v\in\{1,\ldots,p\}$ such that $E_{lk}=E_{ji}\otimes E_{vu}$,}
  \lefteqn{(X\ominus 0_q)\ominus 0_p}\quad &\\
    &= \sum_{i,j=1}^m\sum_{u,v=1}^p\tr((E_{ji}\otimes E_{vu}\otimes\upsilon_{q}^T)X)\tr(\upsilon_p^TE_{uv}))E_{ij} \\
    &= \sum_{i,j=1}^m\tr\left(\left(E_{ji}\otimes \left(\sum_{u,v=1}^p\tr(\upsilon_p^TE_{uv})E_{vu}\right)\otimes\upsilon_{q}^T\right)X\right)E_{ij} \\
    &= \sum_{i,j=1}^m\tr((E_{ji}\otimes \upsilon_p^T\otimes\upsilon_{q}^T)X))E_{ij}.
 \end{align*}
 It follows that for all $i,j\in\{1,\ldots,m\}$ and for all $X\in\mathbb{F}_m\otimes\mathbb{F}_p\otimes\mathbb{F}_q$
 \begin{equation*}
  \tr((E_{ji}\otimes \upsilon_{pq}^T)X)
   = \tr((E_{ji}\otimes \upsilon_p^T\otimes\upsilon_{q}^T)X).
 \end{equation*}
 We conclude that $E_{ji}\otimes \upsilon_{pq}^T=E_{ji}\otimes \upsilon_p^T\otimes\upsilon_{q}^T$ and hence
 \begin{equation*}
  \upsilon_{pq} = \upsilon_p\otimes\upsilon_{q}. \qedhere
 \end{equation*}
\end{proof}

\begin{lemma}
 \label{lem:commvec}%
 Let $\mathbf{0}\neq\mathbf{a}\in\mathbb{F}^2$ and $\mathbf{0}\neq\mathbf{b}\in\mathbb{F}^q$, where $q$ is a prime number.
 Then $\mathbf{a}\otimes\mathbf{b}=\mathbf{b}\otimes\mathbf{a}$ if and only if
 \begin{enumerate}[label=(\alph*)]
  \item $q=2$ and $\mathbf{a}=\beta\mathbf{b}$ for some $\beta\in\mathbb{F}$,
  \item $q\neq 2$ and $\mathbf{a}$ and $\mathbf{b}$ have one of the forms (for some $\beta\in\mathbb{F}$)
        \begin{enumerate}[label=(\roman*)]
         \item $\mathbf{a}=(1,0)^T$, $\mathbf{b}=\beta(1,0,\ldots,0)^T$,
         \item $\mathbf{a}=(0,1)^T$, $\mathbf{b}=\beta(0,\ldots,0,1)^T$,
         \item $\mathbf{a}=(1,1)^T$, $\mathbf{b}=\beta(1,1,\ldots,1)^T$.
        \end{enumerate}
 \end{enumerate}
\end{lemma}

\begin{proof}
 \begin{enumerate}[label=(\alph*)]
  \item
   If $q=2$, then $\mathbf{a}\otimes\mathbf{b}=\mathbf{b}\otimes\mathbf{a}$ if and only if
   \begin{equation*}
    \begin{pmatrix}
     a_1\mathbf{b} \\
     \vdots        \\
     a_p\mathbf{b}
    \end{pmatrix}
    =
    \begin{pmatrix}
     b_1\mathbf{a} \\
     \vdots        \\
     b_p\mathbf{a}
    \end{pmatrix}.
   \end{equation*}
   Since $\mathbf{b}$ has non-zero entries, let $b_j\neq 0$.
   Then $\mathbf{a}=\beta\mathbf{b}$ where $\beta=a_j/b_j$.
  \item
   Let $k\in{1,\ldots,2q}$. Then there exist unique $Q_{q,k},R_{2,k}\in\{1,2\}$
   and $Q_{2,k},R_{q,k}\in\{1,\ldots,q\}$ such that
   \begin{equation*}
    k-1 = (Q_{q,k}-1)q+(R_{q,k}-1) = (Q_{2,k}-1)2+(R_{2,k}-1).
   \end{equation*}
   The entries of $\mathbf{a}$ and $\mathbf{b}$ satisfy
   \begin{equation*}
    (\mathbf{a}\otimes\mathbf{b})_k
      = a_{Q_{q,k}}b_{R_{q,k}}
      = a_{R_{2,k}}b_{Q_{2,k}}
      = (\mathbf{b}\otimes\mathbf{a})_k.
   \end{equation*}
   Given $k\in\{1,\ldots,2q\}$, let $k'$ be defined by
   \begin{equation*}
    k'-1 = (R_{q,k}-1)2+(Q_{q,k}-1).
   \end{equation*}
   Then
   \begin{equation*}
    Q_{q,k}=R_{2,k'},\quad R_{q,k}=Q_{2,k'}.
   \end{equation*}
   It follows that
   \begin{equation}
     \label{eq:chains}
     a_{Q_{q,k'}}b_{R_{q,k'}}
     = a_{R_{2,k'}}b_{Q_{2,k'}}
     = a_{Q_{q,k}}b_{R_{q,k}}
     = a_{R_{2,k}}b_{Q_{2,k}}.
   \end{equation}
   Define an equivalence relation $\sim$ on $\{1,\ldots,2q\}$ generated by
   \begin{equation*}
    k\sim k' \qquad\text{if}\qquad Q_{q,k}=R_{2,k'}\text{~~and~~}R_{q,k}=Q_{2,k'}.
   \end{equation*}
   Furthermore, let $n:\{1,\ldots,2q\}\to\{1,\ldots,2q\}$ be the bijection given by
   \begin{equation*}
    n(k) := (R_{q,k}-1)2+Q_{q,k}.
   \end{equation*}
   Hence, $k\sim n(k)$ and $k\sim k'$ if and only if there exists $j\in\mathbb{N}$ such that $k'=n^j(k)$.
   In particular, $n(1)=1$ and $n(2q)=2q$.
   Let $[k]_\sim$ denote an equivalence class.
   It follows by equation \eqref{eq:chains} that
   $(\mathbf{a}\otimes\mathbf{b})_k = (\mathbf{a}\otimes\mathbf{b})_{k'}$
   for all $k'\in [k]_\sim$. Now suppose $(\mathbf{a}\otimes\mathbf{b})_k\neq 0$. Then
   $a_{Q_{q,k'}},a_{R_{2,k'}},b_{Q_{2,k'}},b_{R_{q,k'}}\neq 0$ for all $k'\in [k]_\sim$ and
   by equation \eqref{eq:chains},
   \begin{equation*}
    \dfrac{a_{R_{2,k'}}}{a_{R_{2,k}}}=\dfrac{b_{Q_{2,k}}}{b_{Q_{2,k'}}},\qquad
    \dfrac{a_{Q_{q,k'}}}{a_{Q_{q,k}}}=\dfrac{b_{R_{q,k}}}{b_{R_{q,k'}}}.
   \end{equation*}
   We have three cases: $a_1,a_2\neq 0$, $a_1\neq0$, $a_2=0$ and $a_1=0$, $a_2\neq 0$.
   Hence,
   \begin{enumerate}[label=(\roman*)]
    \item $a_1\neq0, a_2=0$:
     Since $\mathbf{a}\otimes\mathbf{b}=\mathbf{b}\otimes\mathbf{a}$, \\

     \begin{math}
      \qquad a_1(b_1,b_2,\ldots,b_q,0,\ldots,0,0) = a_1(b_1,0,\ldots,b_{\frac{q+1}{2}},0,\ldots,b_q,0).
     \end{math} \\

     Let $k\in\{2,\ldots,q\}$. Then $Q_{q,k}=1$ and $R_{q,k}=k$ and $n(k)=2(k-1)+1$. If $n(k)>q$,
     then $a_1b_k=(\mathbf{a}\otimes\mathbf{b})_{n(k)}=0$. In general, if $n^{j-1}(k)\leq q$,
     then $n^j(k)=2^j(k-1)+1$. There exists a $J\in\mathbb{N}$ such that $n^J(k)>q$,
     and hence $a_1b_{k}=(\mathbf{a}\otimes\mathbf{b})_{n^J(k)}=0$. It follows that
     \begin{equation*}
      \mathbf{a}=a_1(1,0)^T,\qquad
      \mathbf{b}=b_1(1,0,\ldots,0)^T.
     \end{equation*}
     Taking $\beta=b_1/a_1$ yields form (i).
    \item $a_1=0,a_2\neq 0$:
     We have \\

     \begin{math}
      \qquad a_2(0,0,\ldots,0,b_1,\ldots,b_{q-1},b_q) = a_2(0,b_1,\ldots,0,b_{\frac{q+1}{2}},\ldots,0,b_q).
     \end{math}\\

     Let $k\in\{1,\ldots,q-1\}$. Then $Q_{q,q+k}=2$ and $R_{q,q+k}=k$ and $n(q+k)=2k$.
     If $n(q+k)\leq q$, then $a_2b_k=(\mathbf{a}\otimes\mathbf{b})_{n(q+k)}=0$. In general,
     if $n^{j-1}(q+k)>q$, then $n^j(q+k)=2^jk$. Since the range of $n$ is bounded by $2q$, there
     exists a $J\in\mathbb{N}$ such that $n^J(q+k)\leq q$, and hence
     $a_2b_{k}=(\mathbf{a}\otimes\mathbf{b})_{n^J(q+k)}=0$. It follows that
     \begin{equation*}
      \mathbf{a}=a_2(0,1)^T,\qquad
      \mathbf{b}=b_q(0,\ldots,0,1)^T.
     \end{equation*}
     Taking $\beta=b_q/a_2$ yields form (ii).
    \item $a_1,a_2\neq 0$:
     We have \\
%     \begin{math}
%       (a_1b_1,a_1b_2,a_1b_3,\ldots,a_1b_q,a_2b_1,\ldots,a_2b_{q-2},a_2b_{q-1},a_2b_q) \\
%              = (a_1b_1,a_2b_1,a_1b_2,\ldots,a_1b_{\frac{q+1}2},a_2b_{\frac{q+1}2},\ldots,a_2b_{q-1},a_1b_q,a_2b_q).
%     \end{math} \\
     \hspace*{-1cm}
     \begin{tikzpicture}[anchor=base]
       \matrix[column sep=0.4pt,row sep=1em]
       {
         %\node{}; &
         \node(la1b1){$\bigl(a_1b_1,$}; & 
         \node(la1b2){$a_1b_2,$}; & 
         \node(la1b3){$a_1b_3,$}; & 
         \node(ldots1){$\ldots\rlap{,}$}; &
         \node(la1bq){$a_1b_q,$}; &
         \node(la2b1){$a_2b_1,$}; &
         \node(ldots2){$\ldots\rlap{,}$}; &
         \node(la2bq-2){$a_2b_{q-2},$}; &
         \node(la2bq-1){$a_2b_{q-1},$}; &
         \node(la2bq){$a_2b_q\bigr)$}; \\
         %\node{}; &
         %\node{$=$}; &
         \node(ra1b1){$=\bigl(a_1b_1,$}; & 
         \node(ra2b1){$a_2b_1,$}; & 
         \node(ra1b2){$a_1b_2,$}; & 
         \node(rdots1){$\ldots\rlap{,}$}; &
         \node(ra1bq+1/2){$a_1b_{\frac{q+1}{2}},$}; &
         \node(ra2bq+1/2){$a_2b_{\frac{q+1}{2}},$}; &
         \node(rdots2){$\ldots\rlap{,}$}; &
         \node(ra2bq-1){$a_2b_{q-1},$}; &
         \node(ra1bq){$a_1b_q,$}; &
         \node(ra2bq){$a_2b_q\bigr)$}; \\
       };
       \begin{scope}[->,>=stealth]
         \draw (ra1bq) -- (la2bq-1);
         \draw (la2bq-1) -- (ra2bq-1);
         \draw (ra2bq-1) -- (la2bq-2);
         \draw (ra2bq+1/2) -- (la2b1);
       \end{scope}
       \coordinate (m2) at ($(ldots2)!2/3!(rdots2)$);
       \draw[densely dotted,->,>=stealth] (la2bq-2) -- (m2);
     
       \begin{scope}[->,>=stealth,densely dashed]
         \draw (ra2b1) -- (la1b2);
         \draw (la1b2) -- (ra1b2);
         \draw (ra1b2) -- (la1b3);
         \draw (ra1bq+1/2) -- (la1bq);
       \end{scope}
       \coordinate (m1) at ($(ldots1)!2/3!(rdots1)$);
       \draw[densely dotted,->,>=stealth] (la1b3) -- (m1);
       %\draw (la2b1) .. controls ($(la2b1)!1/2!(ra2b1)$) .. (ra2b1);
     \end{tikzpicture}\\
     and chasing the arrows yields the sequence of equalities
     \begin{align*}
      a_1b_q=a_2b_{q-1}=a_2b_{q-2}=\cdots=a_2b_{\frac{q+1}2}&=a_2b_1 \\
            &=a_2b_1=a_1b_2=a_1b_3=\cdots=a_1b_{\frac{q+1}2}.
     \end{align*}
     It follows that
     $\mathbf{a}\otimes\mathbf{b}
      =(\alpha,\beta,\beta,\ldots,\beta,\beta,\gamma)$
     for some $\alpha,\beta,\gamma\in\mathbb{F}$.
     If $\beta=0$, then $\alpha=0$ or $\gamma=0$ (i.e. to yield a Kronecker product $\mathbf{a}\otimes\mathbf{b}$ of non-zero vectors). Either case contradicts $a_1,a_2\neq 0$.
     Hence $\beta\neq 0$.
     Since $a_1\mathbf{b}=(\alpha,\beta,\ldots,\beta)$ and $a_2\mathbf{b}=(\beta,\ldots,\beta,\gamma)$
     are linearly dependent, $\alpha=\gamma=\beta$.
     Thus we conclude that\\
     \strut\qquad $\mathbf{a}\otimes\mathbf{b} =
       (\beta,\ldots,\beta)
        =(1,1)\otimes \beta(1,\ldots,1).$ \qedhere
   \end{enumerate}
 \end{enumerate}
\end{proof}

\begin{lemma}
 Let $A\in\mathbb{F}_2$ and $B\in\mathbb{F}_q$, where $q$ is a prime number,
 with $\tr(A)=\tr(B)=1$. Let ${J}_{q}$ be the $q\times q$ matrix with every entry equal to 1.
 Matrices with entries 0 or 1 will be assumed to have all necessary entries equal to zero or one
 for the statement to make sense, and in each case the form below is unique. Then $A\otimes B=B\otimes A$ if and only if
 \begin{enumerate}[label=(\alph*)]
  \item $q=2$ and $A=B$,
  \item $q\neq 2$ and $A$ and $B$ have one of the forms (depending on the \\ characteristic of $\mathbb{F}$) \\ %\\[-1.5mm]
  
  \noindent\!\!\!\!\!\!
        \begin{tabular}{@{}r@{\,\,}lr@{\,\,}l@{}}
         (i) & $A=\dfrac{1}{2}I_2$, 
               $B=\dfrac{1}{q}I_q$, &
         (ii) & $A=\dfrac{1}{2}J_2$, 
               $B=\dfrac{1}{q}J_{q}$. \\[5mm]
         (iii) & $A=E_{11}$,
               $B=E_{11}$, &
         (iv) & $A=E_{22}$,
               $B=E_{qq}$, \\[3mm]
         (v) & $A=\begin{pmatrix}1&1\\0&0\end{pmatrix}$, 
               $B=\begin{pmatrix}1&1\\0&0_{q-1}\end{pmatrix}$, &
         (vi) & $A=\begin{pmatrix}0&0\\1&1\end{pmatrix}$, 
               $B=\begin{pmatrix}0_{q-1}&0\\1&1\end{pmatrix}$, \\[4mm]
         (vii) & $A=\begin{pmatrix}1&0\\1&0\end{pmatrix}$, 
               $B=\begin{pmatrix}1&0\\1&0_{q-1}\end{pmatrix}$, &
         (viii) & $A=\begin{pmatrix}0&1\\0&1\end{pmatrix}$, 
               $B=\begin{pmatrix}0_{q-1}&1\\0&1\end{pmatrix}$.
        \end{tabular}
%        \begin{enumerate}[label=(\roman*)]
%         \item $A=\dfrac{1}{2}I_2$, 
%               $B=\dfrac{1}{q}I_q$, \\[1mm]
%         \item $A=\dfrac{1}{2}J_2$, 
%               $B=\dfrac{1}{q}J_{q}$. \\[1mm]
%         \item $A=E_{11}$,
%               $B=E_{11}$, \\[1mm]
%         \item $A=E_{22}$,
%               $B=E_{qq}$, \\[1mm]
%         \item $A=\begin{pmatrix}1&1\\0&0\end{pmatrix}$, 
%               $B=\begin{pmatrix}1&1\\0&0_{q-1}\end{pmatrix}$, \\[2mm]
%         \item $A=\begin{pmatrix}0&0\\1&1\end{pmatrix}$, 
%               $B=\begin{pmatrix}0_{q-1}&0\\1&1\end{pmatrix}$, \\[2mm]
%         \item $A=\begin{pmatrix}1&0\\1&0\end{pmatrix}$, 
%               $B=\begin{pmatrix}1&0\\1&0_{q-1}\end{pmatrix}$, \\[2mm]
%         \item $A=\begin{pmatrix}0&1\\0&1\end{pmatrix}$, 
%               $B=\begin{pmatrix}0_{q-1}&1\\0&1\end{pmatrix}$.
%        \end{enumerate}
 \end{enumerate}
\end{lemma}
\begin{proof}
 \begin{enumerate}[label=(\alph*)]
  \item
   If $q=2$, then $A\otimes B=B\otimes A$ if and only if
   \begin{equation*}
    \begin{pmatrix}
     (A)_{1,1}B & (A)_{1,2}B \\
     (A)_{2,1}B & (A)_{2,2}B
    \end{pmatrix}
    =
    \begin{pmatrix}
     (B)_{1,1}A & (B)_{1,2}A \\
     (B)_{2,1}A & (B)_{2,2}A
    \end{pmatrix}.
   \end{equation*}
   Since $B$ has non-zero entries ($\tr(B)=1$), suppose $(B)_{i,j}\neq 0$.
   Then $A=\alpha B$ where $\alpha=(A)_{i,j}/(B)_{i,j}$.
   Since $1=\tr(A)=\alpha\tr(B)=\alpha$, we have $A=B$.
  \item
   We note that (vii) and (viii) follow by transposition of (v) and (vi) and the fact that $A\otimes B=B\otimes A$ holds if and only if $A^T\otimes B^T=B^T\otimes A^T$.
   Let $\diag(X)$ denote the vector of main diagonal entries of $X$. Then
   \begin{align*}
    \SwapAboveDisplaySkip
    \diag(A)\otimes\diag(B)
      &= \diag(A\otimes B) \\
      &= \diag(B\otimes A) = \diag(B)\otimes\diag(A).
   \end{align*}
   Let $\cdiag(X)$ denote the vector of main counter diagonal entries of matrix $X$. Then, similarly,
   \begin{align*}
    \SwapAboveDisplaySkip
    \cdiag(A)\otimes\cdiag(B)
      %&= \cdiag(A\otimes B) \\
      %&= \cdiag(B\otimes A)
        = \cdiag(B)\otimes\cdiag(A).
   \end{align*}
   Let $\mathbf{a}_1$, \ldots, $\mathbf{a}_2$ be the column vectors of $A$ and let $\mathbf{b}_1$, \ldots, $\mathbf{b}_q$ be the column vectors of $B$. Then
   \begin{align*}
    \lefteqn{
    \begin{pmatrix}
     \mathbf{a}_1\otimes\mathbf{b}_1 & \mathbf{a}_1\otimes\mathbf{b}_2 & \cdots & \mathbf{a}_1\otimes\mathbf{b}_q
     & \mathbf{a}_2\otimes\mathbf{b_1} & \cdots & \mathbf{a}_2\otimes\mathbf{b}_{q-1} & \mathbf{a}_2\otimes\mathbf{b}_q
    \end{pmatrix}}
    \qquad & \\
    &= A\otimes B = B\otimes A  \\
    & =
    \begin{pmatrix}
     \mathbf{b}_1\otimes\mathbf{a}_1 & \mathbf{b}_1\otimes\mathbf{a}_2 & \cdots
     & \mathbf{b}_q\otimes\mathbf{a}_1 & \mathbf{b}_q\otimes\mathbf{a}_2
    \end{pmatrix},\qquad\qquad
   \end{align*}
   or equivalently,
   \begin{equation}
    \label{eq:kroncomm}
    \begin{aligned}
    & \mathbf{a}_1\otimes\mathbf{b}_{2j} = \mathbf{b}_j\otimes \mathbf{a}_2
    && j\in\left\{1,\ldots,\frac{q-1}{2}\right\}, \\
    & \mathbf{a}_1\otimes\mathbf{b}_{2j-1} = \mathbf{b}_{j}\otimes\mathbf{a}_1
    && j\in\left\{2,\ldots,\frac{q+1}{2}\right\}
    \end{aligned}
   \end{equation}
   and similarly for the row vectors.
   Hence, the main counter diagonal, the main diagonal, the first and last column vectors and row vectors of $A$ must each be one of the forms
   from Lemma \ref{lem:commvec}, or zero. Since $\tr(A)=1$, $A$ will have
   one of the forms
   \begin{equation*}
    \begin{pmatrix} 1&0\\ 0&0 \end{pmatrix},
    \begin{pmatrix} 0&0\\ 0&1 \end{pmatrix},
    \begin{pmatrix} 1&1\\ 0&0 \end{pmatrix},
    \begin{pmatrix} 0&1\\ 0&1 \end{pmatrix},
    \begin{pmatrix} 1&0\\ 1&0 \end{pmatrix},
    \begin{pmatrix} 0&0\\ 1&1 \end{pmatrix},
   \end{equation*}
   \begin{equation*}
    \begin{pmatrix} 1&1\\ 1&0 \end{pmatrix},
    \begin{pmatrix} 0&1\\ 1&1 \end{pmatrix},
    \frac12\begin{pmatrix} 1&0\\ 0&1 \end{pmatrix},
    \frac12\begin{pmatrix} 1&0\\ 1&1 \end{pmatrix},
    \frac12\begin{pmatrix} 1&1\\ 0&1 \end{pmatrix},
    \frac12\begin{pmatrix} 1&1\\ 1&1 \end{pmatrix}.
   \end{equation*}
   We have that $\rank(A)\in\{1,2\}$.
   Suppose $\rank(A)=2$, then $A$ has one of the forms
   \begin{equation*}
    \begin{pmatrix} 1&1\\ 1&0 \end{pmatrix},
    \begin{pmatrix} 0&1\\ 1&1 \end{pmatrix},
    \frac12\begin{pmatrix} 1&0\\ 0&1 \end{pmatrix},
    \frac12\begin{pmatrix} 1&0\\ 1&1 \end{pmatrix},
    \frac12\begin{pmatrix} 1&1\\ 0&1 \end{pmatrix}.
   \end{equation*}
   Consider, for example, if $A=\begin{pmatrix} 1&1\\ 1&0 \end{pmatrix}$ then it can be easily
   shown that $B$ must have the form
   \begin{equation*}
    B=\beta\begin{pmatrix}
        1 & 1 & \cdots & 1 & 1 & 1 \\
        1 & * & \cdots & * & 1 & 0 \\
        \vdots & \vdots & \iddots & \iddots & * & 0 \\
        1 & * & \iddots & \iddots & \vdots & \vdots \\
        1 & 1 & * & \cdots & * & 0 \\
        1 & 0 & 0 & \cdots & 0 & 0
       \end{pmatrix}
   \end{equation*}
   (where $*$ denotes an arbitrary value, possibly different for each entry)
   and since the main counter-diagonal of $B$ must consist
   only of 1s by Lemma \ref{lem:commvec}, but the main diagonal
   must be of the form $(1,0,\ldots,0)$ (also by Lemma \ref{lem:commvec}) which shares exactly one entry with the main counter-diagonal ($q>2$ is prime),
%   and comparing the $q$-th component of $\mathbf{a}_1\otimes\mathbf{b}_2$ and $\mathbf{b}_1\otimes\mathbf{a}_2$
%   yields $0\neq\beta$
   no $B$ can be found with $A\otimes B=B\otimes A$ and $\tr(B)=1$.
   All of the remaining examples similarly yield contradictions, except for $A=\frac12I_2$. If $A=\frac12I_2$, then the main diagonal entries of $B$ consists only
   of the single number $\beta$ which yields that $1=\tr(B)=\beta q$, together with
   the forms from Lemma \ref{lem:commvec} for the first and last rows and columns we obtain the form:
   \begin{equation*}
    B=\frac1q\begin{pmatrix}
        1 & 0 & \cdots & 0 & 0 & 0 \\
        0 & 1 & \cdots & * & * & 0 \\
        \vdots & \vdots & \ddots & \iddots & * & 0 \\
        0 & * & \iddots & \ddots & \vdots & \vdots \\
        0 & * & * & \cdots & 1 & 0 \\
        0 & 0 & 0 & \cdots & 0 & 1
       \end{pmatrix}
   \end{equation*}
   %and $B=\beta I_q$, which yields (i).
   Since
   \begin{equation*}
    \frac1{2}
    \begin{pmatrix} \mathbf{b}_2 \\ \mathbf{0} \end{pmatrix}
    =
    \mathbf{a}_1\otimes\mathbf{b}_2
    =
    \mathbf{b}_1\otimes\mathbf{a}_2
    =
    \frac1{q}
    \begin{pmatrix} \mathbf{a}_2 \\ \mathbf{0} \\ \vdots \\ \mathbf{0} \end{pmatrix}
   \end{equation*}
   we obtain $\mathbf{b}_2=\frac1q(0,1,0,\ldots,0)^T$ and, similarly, $\mathbf{b}_{q-1}=\frac1q(0,\ldots,0,1,0)^T$.
   The same holds true for the second and $(q-1)$-th row vectors of $B$.
   Now
   \begin{equation*}
    \mathbf{a}_1\otimes\mathbf{b}_3
    =
    \mathbf{b}_2\otimes\mathbf{a}_1
    =
    \frac1q
    \begin{pmatrix} \mathbf{0} \\ \mathbf{a}_1 \\ \mathbf{0} \\ \vdots \\ \mathbf{0} \end{pmatrix}
    =
    \frac1q
    \begin{pmatrix} 0 \\ 0 \\ 1 \\ 0 \\ \vdots \\ 0 \end{pmatrix}
   \end{equation*}
   which shows that $\mathbf{b}_3=\frac1q(0,0,1,0,\ldots,0)^T$. A similar argument holds for $\mathbf{b}_{q-2}$ and continuing in the same manner, equation \eqref{eq:kroncomm} finally yields that $B=\frac1qI_q$.
   If $\rank(A)=1$, then $A$ has one of the forms
   \begin{equation*}
    \begin{pmatrix} 1&0\\ 0&0 \end{pmatrix},
    \begin{pmatrix} 0&0\\ 0&1 \end{pmatrix},
    \begin{pmatrix} 1&1\\ 0&0 \end{pmatrix},
    \begin{pmatrix} 1&0\\ 1&0 \end{pmatrix},
    \begin{pmatrix} 0&0\\ 1&1 \end{pmatrix},
    \begin{pmatrix} 0&1\\ 0&1 \end{pmatrix},
    \frac12\begin{pmatrix} 1&1\\ 1&1 \end{pmatrix}.
   \end{equation*}
   For example, if $A=\begin{pmatrix} 1&0\\ 0&0 \end{pmatrix}$ then $B$ must have the form
   \begin{equation*}
    B = \beta
    \begin{pmatrix}
     1 & 0 & 0 & \cdots & \cdots & 0 \\
     0 & 0 & * & \cdots & \cdots & * \\
     0 & * & \ddots & \ddots & \cdots & * \\
     \vdots & \vdots & \ddots & \ddots & \ddots & \vdots \\
     0 & * & \cdots & \ddots & \ddots & * \\
     0 & * & \cdots & \cdots & * & 0
    \end{pmatrix}.
   \end{equation*}
   It follows that since $\mathbf{a}_1\neq\mathbf{0}$ and $\mathbf{a}_2=\mathbf{0}$,
   $\mathbf{a}_1\otimes\mathbf{b}_2=\mathbf{b}_1\otimes\mathbf{a}_2$
   if and only if $\mathbf{b}_2=\mathbf{0}$. The equations and their implications are illustrated in the following diagram: \\
   \hspace*{-4mm}
   \begin{tikzpicture}[anchor=base]
     \matrix[column sep=0.4pt,row sep=1em]
     {
       %\node{}; &
       \node[inner sep=2pt](la1b1){$\bigl(\mathbf{a}_1\otimes \mathbf{b}_1,$}; & 
       \node(la1b2){$\mathbf{a}_1\otimes \mathbf{b}_2,$}; & 
       \node(la1b3){$\mathbf{a}_1\otimes \mathbf{b}_3,$}; & 
       \node(ldots1){$\ldots\rlap{,}$}; &
       \node(la1bq){$\mathbf{a}_1\otimes \mathbf{b}_q,$}; &
       \node(la2b1){$\mathbf{a}_2\otimes \mathbf{b}_1,$}; &
       \node(ldots2){$\ldots\rlap{,}$}; &
       \node(la2bq-1){$\mathbf{a}_2\otimes \mathbf{b}_{q-1},$}; &
       \node(la2bq){$\mathbf{a}_2\otimes \mathbf{b}_q\bigr)$}; \\
       %\node{}; &
       %\node{$=$}; &
       \node(ra1b1)[inner sep=2pt]{$\llap{=}\bigl(\mathbf{b}_1\otimes \mathbf{a}_1,$}; & 
       \node(ra2b1){$\mathbf{b}_1\otimes \mathbf{a}_2,$}; & 
       \node(ra1b2){$\mathbf{b}_2\otimes \mathbf{a}_1,$}; & 
       \node(rdots1){$\ldots\rlap{,}$}; &
       \node(ra1bq+1/2){$\mathbf{b}_{\frac{q+1}{2}}\otimes \mathbf{a}_1,$}; &
       \node(ra2bq+1/2){$\mathbf{b}_{\frac{q+1}{2}}\otimes \mathbf{a}_2,$}; &
       \node(rdots2){$\ldots\rlap{,}$}; &
       \node(ra1bq){$\mathbf{b}_q\otimes \mathbf{a}_1,$}; &
       \node(ra2bq){$\mathbf{b}_q\otimes \mathbf{a}_2\bigr)$}; \\
       
       \node(c1){};&
       \node(c2){$\mathbf{b}_2=\mathbf0$};&
       \node(c3){$\mathbf{b}_3=\mathbf0$};&
            &
       \node(c4){};&
       \node(c5){};&
       \node(c6){};&
       \node(c7){$\mathbf{b}_q=\mathbf0$};&
       \node(c8){};\\
     };
     \begin{scope}[<->,>=stealth]
       \draw (ra1bq) -- (la2bq-1);
     \end{scope}
     \coordinate (m2) at ($(ldots2)!2/3!(rdots2)$);
   %  \draw[densely dotted,<->,>=stealth] (la2bq-2) -- (m2);
   
     \begin{scope}[<->,>=stealth]
     \draw(la1b1) -- (ra1b1);
       \draw (ra2b1) -- (la1b2);
       \draw (la1b2) -- (ra1b2);
       \draw (ra1b2) -- (la1b3);
       \draw (ra1bq+1/2) -- (la1bq);
     \end{scope}
     
       \begin{scope}[double equal sign distance]
        \draw[double,-{Implies[]}] (ra2b1) -- (c2);
        \draw[double,-{Implies[]}] (ra1b2) -- (c3);
   %     \draw[double,-{Implies[]}] (ra1bq+1/2) -- (c4);
   %     \draw[double,-{Implies[]}] (ra1bq-1) -- (c6);
        \draw[double,-{Implies[]}] (ra1bq) -- (c7);         
       \end{scope}
     \coordinate (m1) at ($(ldots1)!2/3!(rdots1)$);
   %  \draw[densely dotted,<->,>=stealth] (la1b3) -- (m1);
     %\draw (la2b1) .. controls ($(la2b1)!1/2!(ra2b1)$) .. (ra2b1);
   \end{tikzpicture}
   Continuing in this way, we conclude from \eqref{eq:kroncomm} that
   \begin{equation*}
    \mathbf{a}_1\otimes\mathbf{b}_{2j} = \mathbf{b}_j\otimes \mathbf{a}_2
    \iff \mathbf{b}_{2j} = \mathbf{0},\qquad j\in\left\{1,\ldots,\frac{q-1}{2}\right\}.
   \end{equation*}
   Now $\mathbf{a}_1\otimes\mathbf{b}_3=\mathbf{b}_2\otimes\mathbf{a}_1=\mathbf{0}$ so that $\mathbf{b}_3=\mathbf{0}$.
   Then from $\mathbf{a}_1\otimes\mathbf{b}_5=\mathbf{b}_3\otimes\mathbf{a}_1=\mathbf{0}$ and we find that $\mathbf{b}_5=\mathbf{0}$.
   In general, for $1\leq k \leq \dfrac{q+1}4$,
   \begin{equation*}
    \mathbf{a}_1\otimes\mathbf{b}_{2(2k)-1}=\mathbf{b}_{2k}\otimes\mathbf{a}_1=\mathbf{0} \iff \mathbf{b}_{2(2k)-1}=\mathbf{0},
   \end{equation*}
   and (using the above equation and the equation below inductively,
   since $2k-1<2(2k-1)-1$), for $1\leq k\leq \dfrac{q-1}4$,
   \begin{equation*}
    \mathbf{a}_1\otimes\mathbf{b}_{2(2k-1)-1}=\mathbf{b}_{2k-1}\otimes\mathbf{a}_1=\mathbf{0}
    \iff \mathbf{b}_{2(2k-1)-1}=\mathbf{0}.
   \end{equation*}
   Finally, $\mathbf{0}=\mathbf{a}_2\otimes\mathbf{b}_{q-1}=\mathbf{b}_q\otimes\mathbf{a}_1$ so that $\mathbf{b}_q=\mathbf{0}$.
   Following the diagrammatic argument of Lemma \ref{lem:commvec}(iii), together with the fact that $\mathbf{b}_{2j}=\mathbf{0}$ as above, we find that $\mathbf{b}_j=\mathbf{0}$ for $j\in\{2,\ldots,q\}$. It follows that
   \begin{equation*}
    B = \beta
    \begin{pmatrix}
     1 & 0 & 0 & \cdots & \cdots & 0 \\
     0 & 0 & 0 & \cdots & \cdots & 0 \\
     0 & 0 & \ddots & \ddots & \cdots & 0 \\
     \vdots & \vdots & \ddots & \ddots & \ddots & \vdots \\
     0 & 0 & \cdots & \ddots & \ddots & 0 \\
     0 & 0 & \cdots & \cdots & 0 & 0
    \end{pmatrix}
   \end{equation*}
   which yields (iii). The other rank-1 matrices follow similarly to yield the remaining cases.\qedhere
 \end{enumerate}
\end{proof}

Now we introduce the notation
\begin{equation*}
 x^{\otimes^j} = \underbrace{x\otimes\cdots\otimes x}_{j\text{ terms}}.
\end{equation*}
We have the following analogue to the prime number decomposition theorem.

\begin{corollary}
 Let $\ominus$ be a uniform Kronecker difference given by the pairs $(\upsilon_n,0)$ of matrices,
 $m,n\in\mathbb{N}$, as given in Proposition \ref{prop:assoc}. The following are equivalent:
 \begin{enumerate}[label=(\alph*)]
  \item
   for all $X\in\mathbb{F}_m\otimes\mathbb{F}_p\otimes\mathbb{F}_q$,
   $Y\in\mathbb{F}_q$, $Z\in\mathbb{F}_p$ and $m,p,q\in\mathbb{N}$:
   \begin{equation*}
    (X\ominus Y)\ominus Z = X\ominus (Z\oplus Y),
   \end{equation*}
  \item
   for each prime number $p\in\mathbb{N}$ there exists a unique $\pi_p\in\mathbb{F}_p$ with $\tr(\pi_p)=1$
   such that: if $n\in\mathbb{N}$ has the prime decomposition
   \begin{equation*}
    n = p_1^{j_1}\cdots p_r^{j_r}
   \end{equation*}
   where $r\in\mathbb{N}$, $j_1,\ldots,j_r\in\mathbb{N}$ and $p_1,\ldots,p_r$ are distinct prime numbers,
   then $\upsilon_n$ has the unique decomposition
   \begin{equation*}
    \upsilon_n = \pi_{p_1}^{\otimes^{j_1}}\otimes\cdots\otimes\pi_{p_r}^{\otimes^{j_r}},
   \end{equation*}
   where
   $\mathbb{F}_{n}=\mathbb{F}_{p_1}^{\otimes^{j_1}}\otimes\cdots\otimes\mathbb{F}_{p_r}^{\otimes^{j_r}}$
   and for every product of prime numbers $q_1\cdots q_{k}$ such that $n=q_1\cdots q_{k}$, $k=j_1+\cdots+j_r$,
   \begin{equation*}
    \upsilon_n = \pi_{p_1}^{\otimes^{j_1}}\otimes\cdots\otimes\pi_{p_r}^{\otimes^{j_r}} = \pi_{q_1}\otimes \cdots \otimes \pi_{q_{k}}.
   \end{equation*}
 \end{enumerate}
\end{corollary}

\section{Conclusion}

We have defined Kronecker differences in the obvious way and established fundamental relations between Kronecker quotients and Kronecker differences. Similar to the theory of Kronecker quotients, uniformity of Kronecker differences provides a rich mathematical structure. Studying Kronecker differences via their matrix representations shows strong connections to Kronecker products and sums which reflect the properties of Kronecker sums.

Many open questions remain, in particular: a more concrete connection between classes of Kronecker differences and classes of Kronecker quotients is desirable. Furthermore, Kronecker quotients induced by exponentiating Kronecker differences (over the real numbers)
\begin{equation*}
 \exp(A\ominus B) = \exp(A)\oslash\exp(B)
\end{equation*}
remains unexplored. In this article we have considered only the linear Kronecker quotients and differences, while non-linear Kronecker differences provide a substantially different view which appears to be necessary to study Kronecker quotients induced by exponentiating Kronecker differences.

\bibliographystyle{amsplain}
\bibliography{krondiff}

\appendix

\section{Trace, transpose and orthogonality structures in tensor products}
\label{sec:orthogonality}

This section provides the definitions and technical results which are employed
to prove the theorems in Section \ref{sec:canon}. Technical lemmata are provided
in the appendix, Section \ref{sec:lemmata}. In later results we will make use of the block trace and partial trace as defined in \cite{filipiak18}.

\begin{definition}[Block trace and partial trace]
 \label{def:blocktrace}%
 Denote by $\Btr:\mathbb{F}_n\otimes\mathbb{F}_m\to\mathbb{F}_m$ the linear map
 defined by
 \begin{equation*}
  \Btr(B\otimes C) = \tr(B) C,
  \quad\text{for all $B\in\mathbb{F}_n$, $C\in\mathbb{F}_m$}
 \end{equation*}
 and by $\Ptr:\mathbb{F}_n\otimes\mathbb{F}_m\to\mathbb{F}_n$ the linear map
 defined by
 \begin{equation*}
  \Ptr(B\otimes C) = \tr(C) B,
  \quad\text{for all $B\in\mathbb{F}_n$, $C\in\mathbb{F}_m$}
 \end{equation*}
 and linear extension.
% $(\Btr_{m,n})_{m,n\in\mathbb{N}}$ \qquad OR \qquad $\displaystyle\Btr:\left(\bigoplus_{m,n\in\mathbb{N}}\mathbb{F}_n\otimes\mathbb{F}_m\right)\to\bigoplus_{m\in\mathbb{N}}\mathbb{F}_m$
\end{definition}

We will also require similar definitions for third order tensors, and combinations thereof, which follows.

\begin{definition}
 \label{def:partialtrace}%
 Denote by $\tr_1:\mathbb{F}_m\otimes\mathbb{F}_n\otimes\mathbb{F}_m\to\mathbb{F}_n\otimes\mathbb{F}_m$ the linear map
 given by
 \begin{equation*}
  \tr_1(A\otimes B\otimes C) = \tr(A) B\otimes C,
  \quad\text{for all $A,C\in\mathbb{F}_m$, $B\in\mathbb{F}_n$}
 \end{equation*}
 and linear extension. Similarly, define the maps
 $\tr_2:\mathbb{F}_m\otimes\mathbb{F}_n\otimes\mathbb{F}_m\to\mathbb{F}_m\otimes\mathbb{F}_m$,
 $\tr_3:\mathbb{F}_m\otimes\mathbb{F}_n\otimes\mathbb{F}_m\to\mathbb{F}_m\otimes\mathbb{F}_n$,
 $\tr_{12}:\mathbb{F}_m\otimes\mathbb{F}_n\otimes\mathbb{F}_m\to\mathbb{F}_m$, by
 \begin{align*}
  \tr_1(A\otimes B\otimes C) &= \tr(A) B\otimes C,
  &&\text{for all $A,C\in\mathbb{F}_m$, $B\in\mathbb{F}_n$}, \\
  \tr_2(A\otimes B\otimes C) &= \tr(B) A\otimes C,
  &&\text{for all $A,C\in\mathbb{F}_m$, $B\in\mathbb{F}_n$}, \\
  \tr_3(A\otimes B\otimes C) &= \tr(C) A\otimes B,
  &&\text{for all $A,C\in\mathbb{F}_m$, $B\in\mathbb{F}_n$}, \\
  \tr_{12}(A\otimes B\otimes C) &= \tr(A)\tr(B) C,
  &&\text{for all $A,C\in\mathbb{F}_m$, $B\in\mathbb{F}_n$},
 \end{align*}
 where each map is linearly extended.
\end{definition}

Similar to the block trace and partial trace, we may define a block transpose and partial transpose following the definitions in \cite{bruss05,peres96}.

\begin{definition}[Block transpose and partial transpose]
 Denote the linear map $\BT:\mathbb{F}_n\otimes\mathbb{F}_m\to\mathbb{F}_n\otimes\mathbb{F}_m$ by
 \begin{equation*}
  \BT(B\otimes C) = (B^T)\otimes C,
  \quad\text{for all $B\in\mathbb{F}_n$, $C\in\mathbb{F}_m$}
 \end{equation*}
 and the linear map $\PT:\mathbb{F}_n\otimes\mathbb{F}_m\to\mathbb{F}_n\otimes\mathbb{F}_m$ by
 \begin{equation*}
  \PT(B\otimes C) = B\otimes(C^T),
  \quad\text{for all $B\in\mathbb{F}_n$, $C\in\mathbb{F}_m$}
 \end{equation*}
 and linear extension.
\end{definition}

We will also make use of the transpose in different modes of
the tensor product, similar to the partial transpose, as defined
below.

\begin{definition}
 Denote by
 $T_3:\mathbb{F}_m\otimes\mathbb{F}_n\otimes\mathbb{F}_m\to\mathbb{F}_m\otimes\mathbb{F}_n\otimes\mathbb{F}_m$
 the linear map defined by
 \begin{equation*}
  T_3(A\otimes B\otimes C) = A\otimes B\otimes (C^T),
  \quad\text{for all $A,C\in\mathbb{F}_m$, $B\in\mathbb{F}_n$}
 \end{equation*}
 and linear extension. Similarly, denote by 
 $T_{12}:\mathbb{F}_m\otimes\mathbb{F}_n\otimes\mathbb{F}_m\to\mathbb{F}_m\otimes\mathbb{F}_n\otimes\mathbb{F}_m$
 the map defined by
 \begin{equation*}
  T_{12}(A\otimes B\otimes C) = (A^T)\otimes (B^T)\otimes C,
  \quad\text{for all $A,C\in\mathbb{F}_m$, $B\in\mathbb{F}_n$}
 \end{equation*}
 and linear extension.
\end{definition}

\begin{remark}
 \label{rem:blockispartial}%
 Let $\text{Id}_m:\mathbb{F}_m\to\mathbb{F}_m$ denote the identity on $\mathbb{F}_m$
 and $T:\mathbb{F}_m\to\mathbb{F}_m$ the matrix transpose.
 Then we may write $\tr_1=\tr\otimes\text{Id}_n\otimes\text{Id}_m$,
 $\tr_{12}=\tr\otimes\tr\otimes\text{Id}_m$, $T_3=\text{Id}_m\otimes\text{Id}_n\otimes T$,
 etc. We also have that $\Btr\circ\tr_1=\tr_{12}=\Btr\circ\tr_2$.
\end{remark}

\begin{remark}
 \label{rem:bimodule}%
 The tensor product $\mathbb{F}_m\otimes\mathbb{F}_n$ is a left $\mathbb{F}_m$-module
 when endowed with the product $\cdot:\mathbb{F}_m\times(\mathbb{F}_m\otimes\mathbb{F}_n)\to\mathbb{F}_m\otimes\mathbb{F}_n$
 given by $A\cdot(B\otimes C):=(AB)\otimes C$ and its linear extension. Similarly,
 $\mathbb{F}_m\otimes\mathbb{F}_n$ is a right $\mathbb{F}_m$-module where $(B\otimes C)\cdot A:=(BA)\otimes C$,
 making $\mathbb{F}_m\otimes\mathbb{F}_n$ an $(\mathbb{F}_m, \mathbb{F}_m)$-bimodule. The bilinear form
 $\langle\cdot,\cdot\rangle:\mathbb{F}_n\times\mathbb{F}_n\to\mathbb{F}$ given by
 \begin{equation*}
  \langle A,B\rangle=\tr(AB^T)
 \end{equation*}
 induces the form
 $(\cdot,\cdot):(\mathbb{F}_m\otimes\mathbb{F}_n)\times(\mathbb{F}_m\otimes\mathbb{F}_n)\to \mathbb{F}_m$,
 \begin{equation*}
  (X,Y) = \Ptr(XY^T).
 \end{equation*}
 Now, $\mathbb{F}_m$ is an $\mathbb{F}$-algebra which is isomorphic to its opposite $\mathbb{F}_m^{op}$ (with
 multiplication $A\mathbin{\cdot_{op}} B=B\cdot A$) under the transpose, i.e.
 \begin{equation*}
  (A\cdot B)^T = B^T\cdot A^T = A^T\mathbin{\cdot_{op}}B^T.
 \end{equation*}
 Hence, $\mathbb{F}_m$ is a $(\mathbb{F}_m, \mathbb{F}_m^{op})$-bimodule with
 the left $\mathbb{F}_m$-module product $A\ast B := A\cdot B$ and the right $\mathbb{F}_m^{op}$-module
 product $B\odot A := B\cdot A^T$, in particular
 \begin{equation*}
  (B\odot A)\odot C = (B\cdot A^T)\cdot C^T = B\cdot (A^T\cdot C^T) = B\cdot (C\cdot A)^T = B\odot(A\mathbin{\cdot_{op}}C).
 \end{equation*}
 Now,
 \begin{equation*}
  (A\ast B\odot C)^T = (ABC^T)^T = CB^TA^T = C\ast B^T\odot A.
 \end{equation*}
 This shows that transposition is an involution of the $(\mathbb{F}_m^{op}, \mathbb{F}_m)$-bimodule
 $\mathbb{F}_m$ in the sense of \cite{loos94}. Moreover, we can now show that $(\cdot,\cdot)$
 is an $\mathbb{F}_m$-valued sesquilinear form in the sense of \cite{loos94} (here we are using left modules
 $\mathbb{F}_m\otimes\mathbb{F}_n$ instead of right modules as defined by Loos; however, the concepts
 are equivalent) using Lemma \ref{lem:bisesquilinear}, i.e.
 \begin{equation*}
  (A\cdot X, B\cdot Y) = A\ast (X,Y)\odot B
 \end{equation*}
 which we present as our next lemma. This sesquilinear form induces an orthogonality relation
 on $\mathbb{F}_m\otimes\mathbb{F}_n$ (Lemma \ref{lem:orthomod}).
\end{remark}

\begin{lemma}
 \label{lem:bisesquilinear}%
 Let the function $(\cdot,\cdot):(\mathbb{F}_m\otimes\mathbb{F}_n)\times(\mathbb{F}_m\otimes\mathbb{F}_n)\to \mathbb{F}_m$
 be given by $(X,Y) = \Ptr(XY^T)$. Then for all $X,Y,Z\in\mathbb{F}_m\otimes\mathbb{F}_n$, $A\in\mathbb{F}_m$,
 and $\alpha\in\mathbb{F}$,
 \begin{enumerate}[label=(\alph*)]
  \item $(X+\alpha Y,Z) = (X,Z) + \alpha (Y,Z)$,
  \item $(X,Y+\alpha Z) = (X,Y)+ \alpha (X,Z)$,
  \item $(A\cdot X,Y) = A(X,Y)$,
  \item $(X,A\cdot Y) = (X,Y)A^T$.
 \end{enumerate}
\end{lemma}

\begin{proof}
 (a) and (b) follow by linearity of the transpose and partial trace and distributivity of
 the matrix product.
 Now $X$ and $Y$ may be written in the form
 \begin{equation*}
  X = \sum_{i,j=1}^n X_{ij}\otimes E_{ij}, \qquad
  Y = \sum_{k,l=1}^n Y_{kl}\otimes E_{kl}
 \end{equation*}
 for some matrixes $X_{ij},Y_{kl}\in\mathbb{F}_m$, $i,j,k,l\in\{1,\ldots,n\}$.
 Direct calculation yields
 \begin{align*}
  (X,Y)
    &= \Ptr\left(\left(\sum_{i,j=1}^nX_{ij}\otimes E_{ij}\right)\left(\sum_{k,l=1}^n Y_{kl}\otimes E_{kl}\right)^T\right) \\
    &= \Ptr\left(\sum_{i,j,k,l=1}^n X_{ij}Y_{kl}^T\otimes (E_{ij}E_{lk})\right) \\
    &= \sum_{i,j,k,l=1}^n \tr(E_{ij}E_{lk}) X_{ij}Y_{kl}^T 
     = \sum_{i,j=1}^n X_{ij}Y_{ij}^T,\vphantom{\left(\sum_{i,j,k,l=1}^n\right)}.
 \end{align*}
 and since $A\cdot X = \displaystyle\sum_{i,j=1}^n (AX_{ij})\otimes E_{ij}$,
 \begin{equation*}
  (A\cdot X,Y) = \sum_{i,j=1}^n (AX_{ij})Y_{ij}^T = A\left(\sum_{i,j=1}^n X_{ij}Y_{ij}^T\right) = A(X,Y).
 \end{equation*}
 Similarly,
 \begin{equation*}
  (X,A\cdot Y) = \sum_{i,j=1}^n X_{ij}(AY_{ij})^T = \left(\sum_{i,j=1}^n X_{ij}Y_{ij}^T\right)A^T = (X,Y)A^T.
 \end{equation*}
 Thus, (c) and (d) hold true.
\end{proof}

\begin{example}
 \label{eg:complex}%
 The well known M\"obius matrix representation of complex numbers, given by 
  $a+ib\mapsto\left(\begin{smallmatrix} a & b \\ -b & a\end{smallmatrix}\right)$
 and the complex conjugate
  $a-ib\mapsto\left(\begin{smallmatrix} a & b \\ -b & a\end{smallmatrix}\right)^T$,
 yields a function $\phi:\mathbb{C}_n\to\mathbb{R}_2\otimes\mathbb{R}_n$, namely
 \begin{equation*}
  \sum_{i,j=1}^n (a_{ij}+b_{ij})E_{ij}
   \mapsto \sum_{i,j=1}^n \begin{pmatrix} a_{ij} & b_{ij} \\ -b_{ij} & a_{ij}\end{pmatrix}
                          \otimes E_{ij}.
 \end{equation*}
 Then the sesquilinear form $(\cdot,\cdot)$ on $\mathbb{R}_2\otimes\mathbb{R}_n$
 is equivalent to the Hilbert-Schmidt inner product (also called the
 Frobenius inner product) on $\mathbb{C}_n$, where the Hilbert-Schmidt inner product is given by
 $\langle A,B\rangle=\tr(AB^*)$ and $^*$ denotes the entry-wise complex conjugate of
 the transposed matrix. One easily shows that,
 \begin{equation*}
  \phi(\langle A,B\rangle) = (\phi(A),\phi(B)).
 \end{equation*}
\end{example}

The representation in Example \ref{eg:complex} also expresses orthogonality in $\mathbb{C}_n$,
while in general the sesquilinear form $(\cdot,\cdot)$ induces a more general orthostructure.

\begin{lemma}
 \label{lem:orthomod}%
 Let $\perp$ be the relation on the left $\mathbb{F}_m$-module $\mathbb{F}_m\otimes\mathbb{F}_n$
 given by $X\perp Y$ if and only if $(X,Y)=(Y,X)=0$. Then
 \begin{enumerate}[label=(\alph*)]
  \item $X\perp Y$ $\implies$ $Y\perp X$,
  \item $X\perp Y$ and $X\perp Z$ $\implies$ $X\perp(Y+Z)$,
  \item $X\perp Y$ $\implies$ $X\perp(A\cdot Y)$ for all $A\in\mathbb{F}_m$,
  \item $X\perp Y$ for all $Y\in\mathbb{F}_m\otimes\mathbb{F}_n$ $\iff$ $X=0$.
 \end{enumerate}
\end{lemma}

\begin{proof}
 (a) follows by the definition of $\perp$,
 (b) and (c) follow by Lemma \ref{lem:bisesquilinear}.
 For (d), suppose that $X\perp Y$ for all $Y\in\mathbb{F}_m\otimes\mathbb{F}_n$. In particular,
     Let $Y=I_m\otimes E_{ij}$. Since $X$ may be written in the form
     \begin{equation*}
      X = \sum_{k,l=1}^n X_{kl}\otimes E_{kl}
     \end{equation*}
     for some matrixes $X_{kl}\in\mathbb{F}_m$, $k,l\in\{1,\ldots,n\}$, and since $X\perp I_m\otimes E_{ij}$
     \begin{align*}
      0 &= (X,I_m\otimes E_{ij})
         = \Ptr\left(\left(\sum_{k,l=1}^n X_{kl}\otimes E_{kl}\right)(I_m\otimes E_{ij})^T\right) \\
        &= \Ptr\left(\sum_{k,l=1}^n X_{kl}\otimes (E_{kl}E_{ji})\right)
         = \sum_{k,l=1}^n \tr(E_{kl}E_{ji})X_{kl} \\
        &= X_{ij}.\vphantom{\left(\sum_{k,l=1}^n\right)}
     \end{align*}
     It follows that $X=0$.
\end{proof}

\begin{remark}
 \label{rem:orthomod}%
 Lemma \ref{lem:orthomod} show that the pair $(\mathbb{F}_m\otimes\mathbb{F}_n,\perp)$ is an
 ortho- left $\mathbb{F}_m$-module \cite{piziak92}.
 Moreover, $\mathbb{F}_m\otimes\mathbb{F}_n$ has an orthogonal basis
 \begin{equation*}
  \{\,I_m\otimes E_{ij}\,:\,i,j\in\{1,\ldots,n\}\},
 \end{equation*}
 since
 \begin{equation*}
  (I_m\otimes E_{ij}, I_m\otimes E_{kl}) = \begin{cases} I_m & \text{if~}i=k\text{~and~}j=l,\\ 0 & \text{otherwise}. \end{cases}
 \end{equation*}
 Similarly, we have that $\mathbb{F}_m\otimes\mathbb{F}_n$ is a left $\mathbb{F}_n$ module with
 an orthogonal basis in the obvious way, $\mathbb{F}_m\otimes\mathbb{F}_n\otimes\mathbb{F}_p$ is
 a left $\mathbb{F}_n$-module with an orthogonal basis and so on.
\end{remark}

\begin{lemma}
 \label{lem:orthocmp}%
 Let $X,Y\in \mathbb{F}_m\otimes\mathbb{F}_n$, then $X=Y$ if and only if
 $(X,I_m\otimes E_{ij})=(Y,I_m\otimes E_{ij})$ for all $i,j\in\{1,\ldots,n\}$.
\end{lemma}

\begin{proof}
 Let
 \begin{equation*}
  X = \sum_{i,j=1}^n X_{ij}\otimes E_{ij} = \sum_{i,j=1}^n X_{ij}\cdot (I_m \otimes E_{ij}), \qquad
  Y = \sum_{k,l=1}^n Y_{kl}\cdot (I_m \otimes E_{kl}).
 \end{equation*}
 By Lemma \ref{lem:bisesquilinear} and Remark \ref{rem:orthomod},
 \begin{equation*}
  (X,I_m\otimes E_{ij}) = X_{ij}
 \end{equation*}
 and the lemma follows immediately.
\end{proof}

\section{Additional lemmata}
\label{sec:lemmata}

This section provides the technical results which are employed
to prove the theorems in Section \ref{sec:canon}.

\begin{lemma}
 \label{lem:tracezero}%
 Let $C\in\mathbb{F}_m\otimes\mathbb{F}_n\otimes\mathbb{F}_m$ such that $\tr_2(C)=0$.
 Then $\tr_2(C^T)=0$, $\tr_2(C(A\otimes I_n\otimes B))=0$ and $\tr_{12}(C(A\otimes I_n\otimes B))=0$
 for all $A,B\in\mathbb{F}_m$.
\end{lemma}

\begin{proof}
 We write $C$ in the form
 \begin{equation*}
  C = \sum_{i,j,k,l=1}^m E_{ij}\otimes C_{ijkl}\otimes E_{kl}.
 \end{equation*}
 Then $\tr_2(C) =0$ yields
 \begin{equation*}
  0 = \sum_{i,j,k,l=1}^m \tr(C_{ijkl}) E_{ij}\otimes E_{kl},
 \end{equation*}
 and since $E_{ij}\otimes E_{kl}$, $i,j,k,l\in\{1,\ldots,m\}$, is a basis for $\mathbb{F}_m\otimes\mathbb{F}_m$,
 we have
 \begin{equation*}
  \tr(C_{ijkl}) = 0, \qquad i,j,k,l\in\{1,\ldots,m\}.
 \end{equation*}
 Hence,
 \begin{align*}
  \tr_2(C^T)
   &= \tr_2\left(\sum_{i,j,k,l=1}^m E_{ij}^T\otimes C_{ijkl}^T\otimes E_{kl}^T\right) \\
   &= \sum_{i,j,k,l=1}^m \tr(C_{ijkl}^T)E_{ji}\otimes E_{lk} \\
   &= \sum_{i,j,k,l=1}^m \tr(C_{ijkl})E_{ji}\otimes E_{lk}
    = 0.
 \end{align*}
 Similarly,
 \begin{align*}
  \tr_2(C(A\otimes I_n\otimes B))
   &= \tr_2\left(\sum_{i,j,k,l=1}^m E_{ij}A\otimes C_{ijkl}\otimes E_{kl}B\right) \\
   &= \sum_{i,j,k,l=1}^m  \tr(C_{ijkl})E_{ij}A\otimes E_{kl}B = 0, \\
  \tr_{12}(C(A\otimes I_n\otimes B))
   &= \tr_{12}\left(\sum_{i,j,k,l=1}^m E_{ij}A\otimes C_{ijkl}\otimes E_{kl}B\right) \\
   &= \sum_{i,j,k,l=1}^m \tr(E_{ij}A)\tr(C_{ijkl}) E_{kl}B = 0. \qedhere
 \end{align*}
\end{proof}

%\begin{lemma}
% \label{lem:tracecomm}%
% Let $A\in\mathbb{F}_m\otimes\mathbb{F}_n\otimes\mathbb{F}_m$ and $B\in\mathbb{F}_m\otimes\mathbb{F}_n$. Then
% \begin{equation*}
%  \tr_{12}(A(B\otimes I_m)) = \tr_{12}((B\otimes I_m)A).
% \end{equation*}
%\end{lemma}

%\begin{proof}
% We write $A$ in the form
% \begin{equation*}
%  A = \sum_{i,j=1}^m\sum_{k,l=1}^n E_{ij}\otimes E_{kl}\otimes A_{ijkl},
% \end{equation*}
% so that
% \begin{align*}
%  \tr_{12}(A(B\otimes I_m))
%   &= \tr_{12}\left(\sum_{i,j=1}^m\sum_{k,l=1}^n ((E_{ij}\otimes E_{kl})B)\otimes A_{ijkl}\right) \\
%   &= \sum_{i,j=1}^m\sum_{k,l=1}^n \tr((E_{ij}\otimes E_{kl})B)\otimes A_{ijkl} \\
%   &= \sum_{i,j=1}^m\sum_{k,l=1}^n \tr(B(E_{ij}\otimes E_{kl}))\otimes A_{ijkl} \\
%   &= \tr_{12}\left(\sum_{i,j=1}^m\sum_{k,l=1}^n (B(E_{ij}\otimes E_{kl}))\otimes A_{ijkl}\right) \\
%   &= \tr_{12}((B\otimes I_m)A). \qedhere
% \end{align*}
%\end{proof}

\begin{lemma}
 \label{lem:parttrans1}%
 Let $A\in\mathbb{F}_m\otimes\mathbb{F}_n\otimes\mathbb{F}_m$. Then
 \begin{equation*}
  \tr_{12}(A)^T = \tr_{12}(T_3(A)) = \tr_{12}(A^T), \qquad \tr_3(A)^T = \tr_3(T_{12}(A)).
 \end{equation*}
\end{lemma}

\begin{proof}
 We write $A$ in the form
 \begin{equation*}
  A = \sum_{i,j=1}^m\sum_{k,l=1}^n E_{ij}\otimes E_{kl}\otimes A_{ijkl},
 \end{equation*}
 so that
 \begin{align*}
  \tr_{12}(A)^T
   &= \left(\tr_{12}\left(\sum_{i,j=1}^m\sum_{k,l=1}^n E_{ij}\otimes E_{kl}\otimes A_{ijkl}\right)\right)^T \\
   &= \left(\sum_{i,j=1}^m\sum_{k,l=1}^n \tr(E_{ij})\tr(E_{kl})A_{ijkl}\right)^T \\
   &= \sum_{i,j=1}^m\sum_{k,l=1}^n \tr(E_{ij})\tr(E_{kl})A_{ijkl}^T \\
   &= \tr_{12}\left(\sum_{i,j=1}^m\sum_{k,l=1}^n E_{ij}\otimes E_{kl}\otimes A_{ijkl}^T\right)
    = \tr_{12}(T_3(A)).
 \end{align*}
 The proof that $\tr_3(A)^T = \tr_3(T_{12}(A))$ follows similarly. The equality $\tr_{12}(T_3(A)) = \tr_{12}(A^T)$
 can be seen by noting that $\tr(E_{ij})=\tr(E_{ji})$ in the proof above.
\end{proof}

\begin{lemma}
 \label{lem:parttrans2}%
 Let $A\in\mathbb{F}_m\otimes\mathbb{F}_n\otimes\mathbb{F}_m$. Then
 \begin{equation*}
  \tr_{12}(T_{12}(A)) = \tr_{12}(A),\qquad \tr_3(T_3(A)) = \tr_3(A).
 \end{equation*}
\end{lemma}

\begin{proof}
 The proof is the proof that $\tr(A)=\tr(A^T)$, \textit{mutatis mutandis}.
\end{proof}

\begin{lemma}
 \label{lem:parttrans3}%
 Let $A\in\mathbb{F}_m\otimes\mathbb{F}_n\otimes\mathbb{F}_m$ and $B\in\mathbb{F}_m\otimes\mathbb{F}_n$. Then
 \begin{equation*}
  T_3(A(B\otimes I_m)) = T_3(A)(B\otimes I_m).
 \end{equation*}
\end{lemma}

\begin{proof}
 We write $A$ in the form
 \begin{equation*}
  A = \sum_{i,j=1}^m\sum_{k,l=1}^n E_{ij}\otimes E_{kl}\otimes A_{ijkl},
 \end{equation*}
 so that
 \begin{align*}
  T_3(A(B\otimes I_m))
   &= T_3\left( \sum_{i,j=1}^m\sum_{k,l=1}^n ((E_{ij}\otimes E_{kl})B)\otimes A_{ijkl} \right) \\
   &= \sum_{i,j=1}^m\sum_{k,l=1}^n ((E_{ij}\otimes E_{kl})B)\otimes A_{ijkl}^T \\
   &= \left( \sum_{i,j=1}^m\sum_{k,l=1}^n (E_{ij}\otimes E_{kl})\otimes A_{ijkl}^T \right)(B\otimes I_m)
    = T_3(A)(B\otimes I_m). \qedhere
 \end{align*}
\end{proof}

\begin{lemma}
 \label{lem:parttrequal}%
 Let $A,B\in\mathbb{F}_m\otimes\mathbb{F}_n\otimes\mathbb{F}_m$. Then
 \begin{equation*}
  A=B \quad\iff\quad \tr_{12}(A(X\otimes I_m)) = \tr_{12}(B(X\otimes I_m))
                                   \text{~~for all $X\in\mathbb{F}_m\otimes\mathbb{F}_n$.}
 \end{equation*}
\end{lemma}

\begin{proof}
 Consider $(\mathbb{F}_m\otimes\mathbb{F}_n)\otimes\mathbb{F}_m$ as a left $\mathbb{F}_m$-module as in Remark \ref{rem:bimodule}.
 Then by Lemma \ref{lem:orthocmp}, $A=B$ if and only if $\tr_{12}(A(E_{ij}\otimes I_m)^T)=\tr_{12}(B(E_{ij}\otimes I_m)^T)$
 for all $i,j\in\{1,\ldots,mn\}$.
 ($\implies$) follows by linearity of $\tr_{12}$ and the fact that the matrices $E_{ij}$ form a basis for
 $\mathbb{F}_m\otimes\mathbb{F}_n$. ($\Longleftarrow$) is immediate since each $E_{ij}\in \mathbb{F}_m\otimes\mathbb{F}_n$.
\end{proof}

\begin{lemma}
 \label{lem:trzidz}%
 Let $A\in\mathbb{F}_m\otimes\mathbb{F}_n\otimes\mathbb{F}_p$ and $B\in\mathbb{F}_n\otimes\mathbb{F}_p$.
 If $\tr_{1}(A) = 0$, then $\tr_1(A(I_m\otimes B)) = 0$.
\end{lemma}

\begin{proof}
 We write $A=\displaystyle\sum_{j=1}^r A_{1j}\otimes A_{2j}\otimes A_{3j}$, where $r\in\mathbb{N}$
 and $A_{1j}\in\mathbb{F}_m$, $A_{2j}\in\mathbb{F}_n$ and $A_{3j}\in\mathbb{F}_p$ for $j\in\{1,\ldots,r\}$.
 Then
 \begin{align*}
  \tr_1(A(I_m\otimes B))
   &= \tr_1\left(\sum_{j=1}^r A_{1j}\otimes((A_{2j}\otimes A_{3j})B)\right)
    = \sum_{j=1}^r \tr(A_{1j})(A_{2j}\otimes A_{3j})B \\
   &= \left(\sum_{j=1}^r \tr(A_{1j})(A_{2j}\otimes A_{3j})\right)B
    = \tr_1(A)B = 0. \qedhere
 \end{align*}
\end{proof}

\begin{lemma}
 \label{lem:blockpartial}%
 Let $A\in\mathbb{F}_m\otimes\mathbb{F}_n\otimes\mathbb{F}_p$ and $B\in\mathbb{F}_m$, $C\in\mathbb{F}_n$ and $D\in\mathbb{F}_p$.
 Then
 \begin{align*}
  \tr_{12}(A(B\otimes C\otimes D))
  &=\tr_{12}(A(B\otimes C\otimes I_p))D \\
  &=\Btr(\tr_2(A(B\otimes C\otimes I_p))D \\
  &=\Btr(\tr_1(A(B\otimes C\otimes I_p))D, \\
  \tr_{12}((B\otimes C\otimes D)A)
  &=D\tr_{12}((B\otimes C\otimes I_p)A) \\
  &=D\Btr(\tr_1((B\otimes C\otimes I_p)A) \\
  &=D\Btr(\tr_2((B\otimes C\otimes I_p)A).
 \end{align*}
\end{lemma}

\begin{proof}
 We will prove the first statement, the proof of the second is almost identical.
 Write $A=\displaystyle\sum_{j=1}^r A_{1j}\otimes A_{2j}\otimes A_{3j}$, where $r\in\mathbb{N}$
 and $A_{1j}\in\mathbb{F}_m$, $A_{2j}\in\mathbb{F}_n$ and $A_{3j}\in\mathbb{F}_p$ for $j\in\{1,\ldots,r\}$.
 Then
 \begin{align*}
  \tr_{12}(A(B\otimes C\otimes D))
   &= \tr_{12}\left(\sum_{j=1}^r (A_{1j}B)\otimes (A_{2j}C)\otimes (A_{3j}D)\right) \\
   &= \sum_{j=1}^r \tr(A_{1j}B)\tr(A_{2j}C) A_{3j}D \\
   &= \left(\sum_{j=1}^r \tr(A_{1j}B)\tr(A_{2j}C)A_{3j}\right)D \\
   &= \tr_{12}(A(B\otimes C\otimes I_p))D.
 \end{align*}
 The fact that $\Btr\circ\tr_2=\Btr\circ\tr_1=\tr_{12}$ follows from Remark \ref{rem:blockispartial}.
\end{proof}

\begin{lemma}
 Let $A\in\mathbb{F}_m\otimes\mathbb{F}_n\otimes\mathbb{F}_p$. Then $\tr(\tr_{12}(A)) = \tr(A)$.
\end{lemma}

\begin{proof}
 We write $A=\displaystyle\sum_{j=1}^r A_{1j}\otimes A_{2j}\otimes A_{3j}$, where $r\in\mathbb{N}$
 and $A_{1j}\in\mathbb{F}_m$, $A_{2j}\in\mathbb{F}_n$ and $A_{3j}\in\mathbb{F}_p$ for $j\in\{1,\ldots,r\}$.
 Then
 \begin{align*}
  \tr(\tr_{12}(A))
   &= \tr\left(\sum_{j=1}^r \tr(A_{1j})\tr(A_{2j})A_{3j}\right)
    = \sum_{j=1}^r \tr(A_{1j})\tr(A_{2j})\tr(A_{3j}) \\
   &= \sum_{j=1}^r \tr(A_{1j}\otimes A_{2j}\otimes A_{3j})
    = \tr\left(\sum_{j=1}^r \tr(A_{1j}\otimes A_{2j}\otimes A_{3j}\right) \\
   &= \tr(A).\phantom{\sum_{j=1}^r} \qedhere
 \end{align*}
\end{proof}

\begin{lemma}
 \label{lem:Btrequiv}%
 Let $A,B\in\mathbb{F}_m\otimes\mathbb{F}_n$.
 If $\tr((I_m\otimes X)A)=\tr((I_m\otimes X)B)$ for all $X\in F_n$, then $\Btr(A)=\Btr(B)$.
 If $\tr((Y\otimes I_n)A)=\tr((Y\otimes I_n)B)$ for all $Y\in F_m$, then $\Ptr(A)=\Ptr(B)$.
\end{lemma}

\begin{proof}
 We will prove the first statement, the second follows similarly.
 Write $A=\displaystyle\sum_{i,j=1}^m E_{ij}\otimes A_{ij}$,
 where $A_{ij}\in\mathbb{F}_n$ for $i,j\in\{1,\ldots,m\}$. Similarly, write
 $B=\displaystyle\sum_{i,j=1}^m E_{ij}\otimes B_{ij}$. Now,
 \begin{align*}
  \tr((I_m\otimes X)A)
   &= \tr\left(\sum_{i,j=1}^m E_{ij}\otimes (XA_{ij})\right)
    = \sum_{i,j=1}^m \tr(E_{ij}) \tr(XA_{ij}) \\
   &= \sum_{i=1}^m \tr(XA_{ii}) = \tr\left(X\left(\sum_{i=1}^m A_{ii}\right)\right)
    = \tr(X\Btr(A)).
 \end{align*}
 Hence, for all $X\in\mathbb{F}_n$ we have $\tr(X\Btr(A))=\tr(X\Btr(B))$
 so that $\Btr(A)=\Btr(B)$.
\end{proof}

\end{document}